\documentclass[12pt,a4paper,oneside,reqno]{amsart}%
\usepackage{amsfonts}
\usepackage[T1]{fontenc}
\usepackage{latexsym}
\usepackage{xspace}%
\usepackage{amsmath}%
\usepackage{amssymb}%
\usepackage{graphicx}
\providecommand{\U}[1]{\protect\rule{.1in}{.1in}}
\newtheorem{theorem}{Theorem}
\newtheorem{prop}{Proposition}
\newtheorem{con}{Conjecture}
\newtheorem{ejem}{Example}
\newtheorem{lem}{Lemma}
\newtheorem*{theorem*}{Theorem}
\newtheorem*{prop*}{Proposition}
\newtheorem*{con*}{Conjecture}
\newtheorem*{ejem*}{Example}

\theoremstyle{definition}

\newtheorem{rem}{Remark}

\newcommand{\flr}[1]{\ensuremath{\lfloor #1 \rfloor}}

\usepackage{xcolor}
\usepackage{ifthen}
\newboolean{pedro}
\setboolean{pedro}{true}
\newcommand{\pedro}[1]{\ifthenelse{\boolean{pedro}}{\color{black!40!red}\marginpar{$\star$ #1}\setboolean{pedro}{false}}{\color{black}\setboolean{pedro}{true}}}

\begin{document}
\author{L. Bayón}
\address{Departamento de Matemáticas, Universidad de Oviedo\\
Avda. Calvo Sotelo s/n, 33007 Oviedo, Spain}
\email{bayon@uniovi.es}
\author{P. Fortuny Ayuso}
\address{Departamento de Matemáticas, Universidad de Oviedo\\
Avda. Calvo Sotelo s/n, 33007 Oviedo, Spain}
\email{fortunypedro@uniovi.es}
\author{J. Grau}
\address{Departamento de Matemáticas, Universidad de Oviedo\\
Avda. Calvo Sotelo s/n, 33007 Oviedo, Spain}
\email{grau@uniovi.es}
\author{A. M. Oller-Marcén}
\address{Centro Universitario de la Defensa de Zaragoza - IUMA\\
Ctra. Huesca s/n, 50090 Zaragoza, Spain}
\email{oller@unizar.es}
\author{M. M. Ruiz}
\address{Departamento de Matemáticas, Universidad de Oviedo\\
Avda. Calvo Sotelo s/n, 33007 Oviedo, Spain}
\email{mruiz@uniovi.es}
\begin{abstract}
In this paper, we present a novel method for computing the asymptotic
values of both the optimal threshold, and the probability of success in sequences of optimal
stopping problems. This method, based on the resolution of a first-order
linear differential equation, makes it possible to systematically obtain these
values in many situations. As an example, we address nine variants of the well-known secretary problem, including the classical one, that appear in the literature on the subject, as well as
four other unpublished ones.
\end{abstract}

\title[A new method in optimal stopping problems]{A new method for computing asymptotic results in optimal stopping problems}

\maketitle

\section{Introduction. Optimal Stopping Problems.}
\label{intro}
An \emph{optimal stopping problem} is the task of trying to maximize a payoff function which depends on a sequence of random events by choosing the best moment at which to stop these (so that the payoff function only depends on the past events). The \emph{secretary problem} is perhaps the best-known instance, but examples in the literature are plentiful \cite{FER,gil,101}. These problems can be modeled by means of two finite sequences and a process. The first sequence, $\{X_i\}_{i=1}^n$, consists of random variables with known joint distribution; the other one, $\{P_i\}_{i=1}^n$, is made of functions of $i$ variables $P_{i}(x_1,\ldots,x_i)$, which depend on the observed values of $(X_1,\ldots, X_i)$ for each $i$. The process is stepwise:
\begin{enumerate}
\item At each  step $k$, the value $X_{k}=x_{k}$ is observed. Based on this value, the choice between stopping or continuing is made. If $k=n$, the process stops in any case.
\item The \emph{final payoff} $P_{k}(x_{1},\dots ,x_{k})$ is obtained after stopping.
\end{enumerate}
The objective in these situations is to maximize the expected final payoff which we will denote by $\mathbf{P}$.

If, for instance, $X_{k}\sim Be(p_{k})$ are   mutually independent Bernoulli random variables, and the payoff function
$P_{k}$ only depends on the last observation $x_{k}$ (and not on the previous ones) we can reason as follows: let $\gamma$ be the payoff obtained if the process ends after the $n$-th  step (without stopping), and let $E(k)$ be the expected payoff if the process is not stopped at step $k$ but the optimal strategy is followed from that point on. A recursive argument shows that
\begin{align*}
E(k)  &  =p_{k+1} \max\big\{P_{k+1}(1),E(k+1)\big\}+ (1-p_{k+1})\max
\big\{P_{k+1}(0),E({k+1})\big\},\\
E(n)  &  =\gamma.
\end{align*}
This dynamic program allows us to compute the expected payoff when following the optimal strategy (i.e. $E(0)$) in linear time $\mathcal{O}(n)$, even if we do not actually know which this optimal strategy is.

In the same setting, we can consider the (usual) situation in which the process is never optimal when $X_k=0$ (recall that $X_k$ are Bernoulli), or we are not allowed to stop in that event. This can be modeled setting $P_k(0)=-\infty$. With this condition, the previous recurrence becomes:
\begin{align*}
E(k)  &  =p_{k+1} \max\big\{P_{k+1}(1),E(k+1)\big\}+ (1-p_{k+1})E({k+1}),\\
E(n)  &  =\gamma.
\end{align*}
Defining the \emph{optimal stopping set} $\mathcal{O}$ as
\[
\mathcal{O}=\{k: P_{k}(1)\geq E(k)\},
\]
the optimal strategy in this case consists in stopping whenever $k\in\mathcal{O}$ and $X_k=1$. The expected payoff following this
strategy is precisely $E(0)$.

If the optimal stopping set turns out to be of the form $\mathcal{O}%
=\{\kappa+1,\dots,n\}$, then the number $\kappa$ is called the \emph{optimal stopping
  threshold} and the strategy that consists in stopping at step $\overline{k}$, with
$\overline{k}=\min\{k:k\in\mathcal{O},X_{k}=1\}$, is optimal. It is called the
\emph{optimal threshold strategy}.

Of course, in any problem, we may always decide to follow a threshold strategy using an arbitrary stopping threshold $k$ (not necessarily optimal). If we denote by $\overline
{E}(k)$ the expected payoff obtained when following such a strategy, a recursive argument shows that:
\begin{align*}
\overline{E}(k)  &  =p_{k+1}P_{k+1}(1)+ (1-p_{k+1})\overline
{E}({k+1}),\\
\overline{E}(n)  &  =\gamma.\nonumber
\end{align*}
Obviously, $\overline{E}(\kappa)=\mathbf{P}$ is the maximum expected payoff
using a threshold strategy.

Insofar we have assumed that the number $n$ of events is fixed, but it can be set as a parameter. For any $n$, we may
consider an optimal stopping problem defined by mutually independent Bernoulli random variables
$\{X^{(n)}_{i}\}_{i=1}^{n}$ and payoff functions $\{P_{i}^{(n)}\}_{i=1}^{n}$. Thus, if we assume that the optimal strategy for each $n$ is of the threshold type, we will have a sequence of recursive  functions $\{\overline{E}_n(k)\}_{n}$ representing, for each $n$, the expected payment using the threshold $k$:
\begin{align}
\label{rec}\overline{E}_n(k)  &  =p^{(n)}_{k+1}P^{(n)}_{k+1}(1)+ (1-p^{(n)}_{k+1})\overline
{E}_n({k+1}),\\
\overline{E}_n(n)  &  =\gamma_n.\nonumber
\end{align}
Each of these problems will have an optimal stopping threshold $\kappa_{n}$,
and an expected payoff $\overline{E}_{n}(\kappa_{n})=\mathbf{P}_{n}$ using the
corresponding optimal threshold strategy. A natural problem is thus to study the asymptotic behavior of these values as $n$ tends
to infinity. Specifically, computing $\lim_{n} \frac{\kappa_{n}}{n}$ and $\lim_{n}
\mathbf{P}_{n}$.

In many cases, the optimal stopping threshold
$\kappa_{n}$ happens to be asymptotically of the form $\kappa_{n}\sim
n\theta$ for some $\theta\in[0,1]$. Recall that $\kappa_{n}$ is, by  definition, the value for which the function $\overline{E}_{n}$ reaches its maximum.
The computation of $\theta=\lim_n \frac{\kappa_n}{n}$ can be achieved, under adequate conditions, by means of the following sequence of functions $f_n:[0,1]\rightarrow \mathbb{R}$
\begin{equation*}
  f_n(x) = \overline{E}_n(\lfloor nx \rfloor).
\end{equation*}
If, for example, 
$f_n(x)$ converges uniformly to a continuous function $f\in C[0,1]$ with a single global maximum $\theta$, then we have shown that \cite{mejorpeor}:
\begin{equation*}
  \lim_{n}\frac{\kappa_n}{n} = \theta, \,\,\, \lim_n \mathbf{P}_n = f(\theta).
\end{equation*}
The uniform convergence of the sequence $f_n(x)$ to a continuous function is an issue which can often be \emph{heuristically} ascertained, but whose proof needs not be straightforward at all. In \cite{yo} we were able to prove the following result, when $\overline{E}_n(k)$ satisfies a recurrence relation similar to the one above:

\begin{theorem}
\label{inicial} Consider a sequence of functions $F_{n}:[0,n]\cap
\mathbb{Z}\rightarrow\mathbb{R}$, each of which defined recursively by the
conditions:
\[
F_{n}(k)=G_{n}(k)+H_{n}(k)F_{n}(k+1)\text{ and }F_{n}(n)=\mu.
\]
Let $f_{n}(x):=F_{n}(\lfloor{nx}\rfloor)$, $h_{n}(x):=n(1-H_{n}(\lfloor
{nx}\rfloor))$ and $g_{n}(x):=nG_{n}(\lfloor{nx}\rfloor)$. If both $h_{n}(x)$
and $g_{n}(x)$ converge in $(0,1)$ and uniformly in $[\varepsilon
,\varepsilon^{\prime}]$ for all $0<\varepsilon<\varepsilon^{\prime}<1$ to
continuous functions in $(0,1)$, $h(x)$ and $g(x)$, respectively, and
$f_{n}(x)\rightarrow f(x)$ uniformly in $[0,1]$ with $f\in C[0,1]$, then
$f(1)=\mu$ and $f$ satisfies the following differential equation in $(0,1)$
\[
f^{\prime}(x)= f(x)h(x)-g(x).
\]
\end{theorem}

This is in fact a very useful result \cite{nosotros3, yo}, but the important issue about the uniform convergence remains. The aim of this paper is to overcome this complication by showing how, under certain conditions on $F_n(k), G_n(k)$ and $H_n(k)$ in the theorem above, uniform convergence is guaranteed and the functions $f(x), g(x), h(x)$ satisfy the differential equation of the statement. We demonstrate the power of our result by revisiting a number of well-known
problems, as well as addressing some new ones, and applying it to them.

Solving differential equations in order to determine asymptotic values in optimal stopping problems and, more specifically, in variants of the secretary problem has numerous precedents \cite{coste1,arxi,Erik,Samuels3,Samuels1,lorenzen,mucci1,mucci2,Samuels2,Tamaki}. This is not surprising, given the relationship between difference and differential equations. The importance of the present work lies in providing a systematic methodology for all the variants of the secretary problem in the literature, and for optimal stopping problems of similar nature. Certainly, the technique is also applicable to a great variety of sequences of recurring functions.

Section 2 is dedicated  to the main result. The long Section 3 is devoted to applying our methodology to several variants of the secretary problem, all of them well-known, in a unified way: the original \emph{secretary problem}  \cite{gil,101}, the \emph{postdoc} variant \cite{mejorpeor,aesima,posdoc2021,posdoc}, the \emph{best-or-worst} version \cite{mejorpeor,nosotros2}, the \emph{secretary problem with uncertain employment}  \cite{refusal}, the \emph{secretary problem with interview cost} \cite{coste1}, the \emph{win-lose-or-draw marriage problem} \cite{ferguson}, the \emph{duration problem} \cite{FER2}, the \emph{multicriteria secretary problem} \cite{multiatri}, and the \emph{secretary problem with a random number of applicants} \cite{sonin,ras}. Section 4 also includes applications of the new methodology, but now to other problems created \emph{ad hoc} such as \emph{lotteries with increasing prize}, the \emph{secretary problem with wildcard}, the \emph{secretary problem with random interruption of the interviews}, and the \emph{secretary problem with penalty if the second best is selected}. Finally, in Section 5 we present and motivate two lines of continuation of this research: on one side, stopping problems whose optimal strategy involves several thresholds and, on the other, sequences of recurrent functions $F_n:\{0,...,n\} \longrightarrow \mathbb{R}$ for which the sequence $ f_n(x):= F(\lfloor nx\rfloor)$ does not converge uniformly in $[0,1]$, but does so punctually in $(0,1)$.

\section{The main result.}

This section is devoted to proving our main result. In forthcoming sections, we will use it to establish a novel methodology for determining the asymptotic
optimal threshold and the asymptotic probability of success in problems for
which the optimal strategy is a threshold strategy. As we already mentioned,
the underlying ideas were present in \cite{yo}. The following two technical lemmas are easy but helpful.

\begin{lem}
\label{lem1p} Let $f:[0,1]\longrightarrow\mathbb{R}$ be a continuous function and, for every $n$, let $\widetilde{f}_{n}(x)=\displaystyle f\left(
  \frac{\lfloor nx\rfloor}{n}\right)  :[0,1]\longrightarrow\mathbb{R}$. Then, the sequence of functions $\{\widetilde{f}_{n}\}$ converges uniformly to $f$
on $[0,1]$.
\end{lem}
\begin{proof}
Since $[0,1]$ is compact, $f$ is uniformly continuous in $[0,1]$. For every $x\in [0,1]$,
\[
0\leq\left|  \frac{\lfloor nx\rfloor}{n}-x\right|  <\frac{1}{n}%
\]
so the uniform continuity of $f$ in $[0,1]$ gives the result.
\end{proof}

\begin{lem}
\label{imp} Let $\{S_{n}\}$ be a sequence of functions $S_{n}:\{0,\dots,n\}\longrightarrow
\mathbb{R}$ recursively defined as:
\begin{align*}
&  S_{n}(n)=a_{n}\\
&  S_{n}(n-1)=b_{n}\\
&  S_{n}(k)=T_{n}(k)+U_{n}(k)S_{n}(k+1),\ 1\leq k\leq n-2\\
&  S_{n}(0)=c_{n}%
\end{align*}
for some $a_n,b_n, c_n\in \mathbb{R}$, and functions $T_{n},U_{n}:\{0,\dots,n\}\longrightarrow\mathbb{R}$. For $n\in \mathbb{N}$, define $s_{n}:[0,1]\longrightarrow\mathbb{R}$ as
$s_{n}(x)=S_{n}(\lfloor nx\rfloor)$, and $t_{n}=\sum
_{k=1}^{n-2}|T_{n}(k)|$.

If $\lim_{n} a_{n}=\lim_{n} b_{n}=\lim c_{n}=\lim_{n}
t_{n}=0$ and $|U_{n}(k)|\leq1$, then the sequence of functions $\{s_{n}\}$
converges uniformly to $0$ in $[0,1]$.
\end{lem}

\begin{proof}
By recurrence, for $k\in \left\{ 1,\ldots, n-2 \right\}$, we have
\[
S_{n}(k)=b_{n}\prod_{i=2}^{n-k}U_{n}(n-i)+\sum_{i=2}^{n-k}\left(
{\small \prod_{j=i+1}^{n-k}U_{n}(n-j)}\right)  T_{n}(n-i).
\]

Taking into account that $0\leq \lfloor nx \rfloor \leq n$ for $x\in[0,1]$, we get
\[
|s_{n}(x)|=|S_{n}(\lfloor nx\rfloor)|\leq|a_{n}|+|b_{n}|+|c_{n}|+t_{n}%
\]
and the result follows from Lemma \ref{lem1p}.
\end{proof}

\begin{rem}
\label{remk0} Notice that even removing $\lim_{n} a_{n}=0$ and
$\lim_{n} c_{n}=0$ in Lemma \ref{imp}, we can still prove that $\lim_{n}
S_{n}(k)=0$ for every $1\leq k\leq n$. This also remains true if we furthermore
replace the condition $\lim_{n} t_{n}=0$ by the weaker condition $\lim_{n}
T_{n}(k)=0$ for every $1\leq k\leq n$. Thus, whatever $a_n$ and $c_n$ are, the sequence $s_n(x)$ converges uniformly to $0$ in $[\epsilon, 1-\epsilon]$ for any $\epsilon>0$.
\end{rem}

We can now prove our main result

\begin{theorem}
\label{main} Let $\mu\in \mathbb{R} $ be a constant real number and $\{F_{n}\}$ a sequence of functions $F_{n}:\{0,\dots,n\}\longrightarrow
\mathbb{R}$ recursively defined as
\begin{align*}
&  F_{n}(n)=\mu,\\
&  F_{n}(k)=G_{n}(k)+H_{n}(k)F_{n}(k+1),\ 0\leq k\leq n-1,
\end{align*}
for some functions $G_{n},H_{n}:\{0,\dots,n\}\longrightarrow\mathbb{R}$.

For every $n\in \mathbb{N}$, let $f_{n},g_{n},h_{n}:[0,1]\longrightarrow
\mathbb{R}$ be the functions $f_{n}(x)=F_{n}(\lfloor nx\rfloor)$, $h_{n}(x)=n(1-H_{n}%
(\lfloor nx\rfloor))$, and $g_{n}(x)=nG_{n}(\lfloor nx\rfloor)$, respectively. Assume the following conditions hold:

\begin{enumerate}
\item $|H_{n}(k)|\leq1$.

\item $\lim_{n} \big(G_{n}(n-1)+\mu H_{n}(n-1)\big)=\mu$.

\item There exist $h,g\in C^{1}(0,1)$ such that the differential
equation $y^{\prime}=yh-g$ admits a solution $f\in C[0,1]$ with:

\begin{itemize}
\item[(i)] $f(1)=\mu$,

\item[(ii)] $\lim_{n} \big(G_{n}(0)+f(0)H_{n}(0)\big)=f(0)$.

\item[(iii)] $\lim_{n} \frac{1}{n}\sum_{k=1}^{n-2}|V_{n}(k)|=0$, where
\[
V_{n}(k)=\left(  g_{n}\left(  \frac{k}{n}\right)  -g\left(  \frac{k+1}%
{n}\right)  \right)  -f\left(  \frac{k+1}{n}\right)  \left(  h_{n}\left(
\frac{k}{n}\right)  -h\left(  \frac{k+1}{n}\right)  \right)  .
\]

\item[(iv)] $\lim_{n} \sum_{k=1}^{n-2}\frac{M_{n}(k)}{n^{2}}=0$, where $M_n(k)$ is given by
$$M_{n}(k)=\max\{|f^{\prime\prime}(x)|:x\in[k/n,(k+1)/n]\}.$$
\end{itemize}
\end{enumerate}

Then, the sequence of functions $\{f_{n}\}$ converges uniformly to $f$ on
$[0,1]$.
\end{theorem}

\begin{proof}
By definition $f\in C^{2}(0,1)$, so that Taylor's theorem ensures that for each $k\in\{1,\ldots,n-2\}$, there exists $c_n(k)\in (k/n, (k+1)/n)$ such that:
\[
f\left(  \frac{k}{n}\right)  =f\left(  \frac{k+1}{n}\right)  -\frac{1}%
{n}f^{\prime}\left(  \frac{k+1}{n}\right)  +\frac{1}{2n^{2}}f^{\prime\prime
}\left(  c_{n}(k)\right)  ,
\]
On the other hand, since $f$ satisfies
the differential equation $y^{\prime}=yh-g$ in $(0,1)$, then for
 $k\in \left\{ 1,\ldots, n-2 \right\}$ the above equality can be rewritten as
\[
f\left(  \frac{k}{n}\right)  =f\left(  \frac{k+1}{n}\right)  -\frac{1}%
{n}\left(  f\left(  \frac{k+1}{n}\right)  h\left(  \frac{k+1}{n}\right)
-g\left(  \frac{k+1}{n}\right)  \right)  +\frac{1}{2n^{2}}f^{\prime\prime
}\left(  c_{n}(k)\right)  .
\]

Define, for each $n$, the function $S_{n}:\{0,\dots,n\}\longrightarrow
\mathbb{R}$ as $S_{n}(k)=F_{n}(k)-f\left(  \frac{k}{n}\right)  $. Certainly, the following equalities hold:
\begin{align*}
&  S_{n}(n)=F_{n}(n)-f(1)=0\\
&  S_{n}(n-1)=F_{n}(n-1)-f\left(  \frac{n-1}{n}\right)  =G_{n}(n-1)+\mu
H_{n}(n-1)-f\left(  \frac{n-1}{n}\right) \\
&  S_{n}(k)=\left(  \frac{1}{n}V_{n}(k)-\frac{1}{2n^{2}}f^{\prime\prime}\left(
c_{n}(k)\right)  \right)  +H_{n}(k+1)S_{n}(k+1)\\
&  S_{n}(0)=F_{n}(0)-f(0)
\end{align*}
In order to apply Lemma \ref{imp}, we need to check that
$\lim_{n} S_{n}(0)=0$. To do so, just observe that
\[
F_{n}(0)-f(0)=\big(G_{n}(0)+f(0)H_{n}(0)-f(0)\big)+H_{n}(0)S_{n}%
(1)+H_{n}(0)\big(f(1/n)-f(0)\big),
\]
noting (recall Remark \ref{remk0}) that $\lim_{n} S_{n}(1)=0$, $H_{n}(0)$ is
bounded, and $f\in C[0,1]$. Since the remaining hypothesis of Lemma \ref{imp}
follow immediately from conditions (1)-(3) above, we conclude that $\{s_{n}\}$
converges to 0 uniformly in $[0,1]$.

Now, $f_{n}(x)=F_{n}(\lfloor nx\rfloor)=s_{n}(x)+f\left(  \frac{\lfloor
nx\rfloor}{n}\right)  $. Since $f\in C[0,1]$, Lemma \ref{lem1p} implies that
$f\left(  \frac{\lfloor nx\rfloor}{n}\right)  $ converges uniformly to $f$ on
$[0,1]$ and the result follows.
\end{proof}

\begin{rem}
\label{remgh} As suggested by the expression of $V_{n}(k)$, the most readily
available candidates for $g$ and $h$ are the functions defined as the
(pointwise) limits of the sequences $\{g_{n}\}$ and $\{h_{n}\}$. Namely,
\begin{align*}
g(x)  &  :=\lim_{n} g_{n}(x)=\lim_{n} nG_{n}(\lfloor nx\rfloor),\\
h(x)  &  :=\lim_{n} h_{n}(x)=\lim_{n} n(1-H_{n}(\lfloor nx\rfloor)).
\end{align*}
Note that this construction may not lead to $g,h\in C^{1}(0,1)$.
However, the latter property will hold in most of the forthcoming examples.
\end{rem}

\section{Application to known problems}
In this section, we are going to apply Theorem \ref{main} to a collection of some well-known problems in order to illustrate the usefulness of our result, and to show how all those problems can be dealt with in a systematic way using our technique. Recall from the Introduction that $n$ is the number of independent events (sequential choices), $X_i$ are mutually independent Bernouilli random variables (whose $p_i$ are possibly different), $\mathbf{P}_n$ denotes the expected payoff under the optimal threshold strategy and $\kappa_n$ is the optimal stopping threshold. In all cases, there is a sequence of functions $\{F_n\}$ with $F_n:\{0,\dots,n\}\longrightarrow \mathbb{R}$, defined recursively as:
\begin{align*}
&  F_{n}(n)=\mu,\\
&  F_{n}(k)=G_{n}(k)+H_{n}(k)F_{n}(k+1),\ 0\leq k\leq n-1,
\end{align*}
where $G_n(k)=p_{k+1}^{(n)}P^{(n)}_{k+1}(1)$ and $H_n(k)=1-p^{(n)}_{k+1}$. The following two properties characterize $\kappa_n$:
\begin{enumerate}
\item It maximizes $F_n$, that is: $\textbf{P}_n=F_n(\kappa_n)=\max \{F_n(k):0\leq k\leq n\}$, and
\item It is the largest value for which it is preferable to continue rather than to stop:
$$
F_n(\kappa_n)>P^{(n)}_{\kappa_n}(1),\textrm{ and } F_n(\kappa_n+i)\leq P^{(n)}_{\mathbf{\kappa}_n+i}(1)\textrm{ for }1\leq i\leq n-\kappa_n.
$$
\end{enumerate}

These two properties allow us to apply the following two technical results to perform the desired asymptotic analysis.
\begin{prop}
\label{conv} Let $F_{n}%
:\{0,...,n\}\longrightarrow\mathbb{R}$ be a sequence of functions and  $\mathcal{M}(n)$ an argument for which $F_{n}$ is maximum. Define $\{f_{n}\}_{n\in\mathbb{N}}$ as $f_{n}%
(x):=F_{n}(\lfloor nx\rfloor)$, and assume that $\{f_n\}$ converges uniformly in $[0,1]$ to $f\in C[0,1]$ having a single global maximum $\theta$ in $[0,1]$. Then

\begin{itemize}
\item[i)] $\displaystyle\lim_{n} \mathcal{M}(n)/n =\theta$.

\item[ii)] $\displaystyle\lim_{n} F_{n}(\mathcal{M}(n))= f(\theta)$.
\end{itemize}
\end{prop}

\begin{proof} See \cite{mejorpeor}.
\end{proof}

\begin{prop}
\label{nuevaa} Let $\{F_{n},Q_n\}_{n\in\mathbb{N}}$   be two sequences of real functions defined in $\{0,\dots,n\}$ and let $\mathcal{N}(n)\in\{0,\dots,n-1\}$ be such that
\begin{align*}
Q_n(\mathcal{N}(n))&< F_{n}\left(\mathcal{N}(n) \right),\\
Q_n({\mathcal{N}(n)+i})&\geq F_{n}\left(\mathcal{N}(n)+i \right)\textrm{ for all } i=1,...,n-\mathcal{N}(n).
\end{align*}
Assume that the sequences of functions $\{f_{n}\}_{n\in\mathbb{N}}$ and $\{q_{n}\}_{n\in\mathbb{N}}$
defined by $f_{n}(x)=F_{n}(\lfloor nx\rfloor)$ and $q_{n}(x)=Q_{n}(\lfloor nx\rfloor)$ for $x\in[0,1]$ converge uniformly in $[0,1]$ to continuous functions $f$ and $q$ (respectively), and assume there is a unique $\theta\in(0,1]$ such that  $q(x)-f(x)$ changes sign around $\theta$. Then,
$\displaystyle\lim_{n}\mathcal{N}(n)/n=\theta$.
\end{prop}
\begin{proof}
  By the uniform continuity, if there is such $\theta$, then it is unique under the conditions on $Q_n$ and $F_n$. Let $\epsilon>0$ be such that $q(x)<f(x)$ for $x\in[\theta-\epsilon, \theta)$ and $q(x)>f(x)$ for $x\in(\theta, \theta+\epsilon]$. Define the new sequences
  \begin{equation*}
    \overline{Q}_n(k) = \left\{
      \begin{array}{ll}
        Q_n(\mathcal{N}(n) - \lfloor\epsilon n\rfloor) &\mbox{ if } k<\mathcal{N}(n) - \lfloor \epsilon n\rfloor\\[0.3em]
        Q_n(k) & \mbox{ if } \mathcal{N}(n) - \lfloor \epsilon n\rfloor \leq k \leq \mathcal{N}(n) + \lfloor \epsilon n\rfloor\\[0.3em]
        Q_n(\mathcal{N}(n) + \lfloor\epsilon n\rfloor) & \mbox{ if } k>\mathcal{N}(n) + \lfloor \epsilon n\rfloor
      \end{array}
    \right.
  \end{equation*}
  and
  \begin{equation*}
    \overline{F}_n(k) = \left\{
      \begin{array}{ll}
        F_n(\mathcal{N}(n) - \lfloor\epsilon n\rfloor) &\mbox{ if } k<\mathcal{N}(n) - \lfloor \epsilon n\rfloor\\[0.3em]
        F_n(k) & \mbox{ if } \mathcal{N}(n) - \lfloor \epsilon n\rfloor \leq k \leq \mathcal{N}(n) + \lfloor \epsilon n\rfloor\\[0.3em]
        F_n(\mathcal{N}(n) + \lfloor\epsilon n\rfloor) & \mbox{ if } k>\mathcal{N}(n) + \lfloor \epsilon n\rfloor
      \end{array}
    \right.
  \end{equation*}
  These sequences converge uniformly to $q(x),f(x)$ for $x\in[\theta-\epsilon, \theta+\epsilon]$ and to the values $q(\theta-\epsilon)$, $f(\theta-\epsilon)$, (and $q(\theta+\epsilon)$, $f(\theta+\epsilon)$) for $x\leq \theta-\epsilon$, (and $x\geq \theta+\epsilon$), respectively. By the continuity of $q(x)$ and $f(x)$, the function defined in $[0,1]$ by $h(x)=1-(q(x)-f(x))^2$ has a single maximum at $\theta$. The sequence of functions $H_n=\{1-(\overline{Q}_n-\overline{F}_n)^2\}$ converges uniformly to $h(x)$ in $[0,1]$. The result follows now from Proposition \ref{conv}.
\end{proof}

In what follows, each problem is succinctly stated and we will make extensive use of Theorem 2, and Propositions \ref{conv} and \ref{nuevaa}. The required conditions are stated without explanation when they are easy to verify.

\subsection{The Classical Secretary Problem.}

An employer is willing to hire the best one
of $n$ candidates, who can be ranked somehow. They are interviewed one by one in
random order and a decision about each particular candidate has to be made
immediately after the interview, taking into account that, once rejected, a
candidate cannot be called back. During the interview, the employer can rank
the candidate among all the preceding ones, but is unaware of the rank of the
yet unseen candidates. The goal is to determine the optimal strategy that
maximizes the probability of successfully selecting the best candidate.

This problem is an optimal stopping one with a threshold optimal strategy \cite{ods,FER,gil,101}
that consists in choosing the first maximal candidate interviewed after the optimal
threshold. Using the notation and terminology from Section 1, $X^{(n)}_{k}=1$ if and only
if the $k$-th candidate is better than all the previous ones; so
$p^{(n)}_{k}=\frac{1}{k}$, and the payoff function is $P^{(n)}_{k}(1)=\frac{k}{n}$, since
$\frac{k}{n}$ is precisely the probability of success if we choose the $k$-th candidate
provided it is maximal at that step. The expected payoff using a threshold strategy (with
threshold $k$) is equal to the probability of successfully choosing the best candidate
using such strategy. Thus, if we denote this probability by $F_{n}(k)$, it follows from
(\ref{rec}) that the functions $F_{n}(k)$ satisfy the following recurrence relation:
\begin{align*}
F_{n}(k)  &  =\frac{1}{n}+ \frac{k}{k+1}F_{n}(k+1),\\
F_{n}(n)  &  =0.
\end{align*}
and the objective is to maximize this probability.

With the notation of Theorem \ref{main} we have that $$\mu=0, G_{n}(k)   =\frac{1}{n},\textrm{ and } H_{n}(k)=\frac{k}{k+1}.$$
so that
\begin{align*}
g_{n}(x)  &  =nG_{n}(\lfloor nx\rfloor)=1,\\
h_{n}(x)  &  =n(1-H_{n}(\lfloor nx\rfloor))=\frac{n}{\lfloor nx\rfloor+1},
\end{align*}
and taking into account Remark \ref{remgh}, we can consider
\begin{align*}
g(x)  &  =\lim_{n}g_{n}(x)=1,\\
h(x)  &  =\lim_{n}h_{n}(x)=\frac{1}{x}.
\end{align*}
Thus, $f(x)$ is the solution of the IVP
\[
\left\{
\begin{array}
[c]{c}%
y^{\prime}=\dfrac{y}{x}-1\\
\multicolumn{1}{l}{y(1)=0}%
\end{array}
\right.
\]
which gives:
\[
f(x)=-x\log x.
\]
The hypotheses of Theorem \ref{main} hold:
\begin{itemize}
\item Conditions (1), (2), (3i) and (3ii) are straightforward.

\item Condition (3iii) holds because $V_{n}(k)=0$.

\item Condition (3iv) follows because $M_{n}(k)=n/k$, as
$|f^{\prime\prime}(x)|=1/x$ is decreasing.
\end{itemize}

Applying Theorem \ref{main}, $F_{n}(\lfloor
nx\rfloor)$ converges uniformly to $f(x)=-x\log x$ in $[0,1]$. Hence, since
$f(x)$ reaches its maximum at $x=e^{-1}$ and $f(e^{-1})=e^{-1}$,
Proposition \ref{conv} gives the well-known results:
\begin{align*}
&  \lim_{n} \frac{\kappa_{n}}{n}=e^{-1}\\
&  \lim_{n} \mathbf{P}_{n}=e^{-1}%
\end{align*}

\subsection{The Postdoc variant.}
This problem is essentially the previous one with the difference that
the goal is to select the second best candidate. We know
\cite{mejorpeor, Rose,posdoc2021} that the probability $F_{n}(k)$ of
successfully choosing the second best candidate using a threshold
strategy with threshold $k$ satisfies:
\begin{align*}
F_{n}(k)  &  =\frac{k}{n(n-1)}+ \frac{k}{k+1}F_{n}(k+1),\\
F_{n}(n)  &  =0.
\end{align*}
Thus, the relevant data are given in Table 1
\begin{table}[h]
\caption{Data for the Postdoc variant.}%
\label{Tpost}%
\begin{tabular}
[c]{|c|c|c|c|c|c|c|}\hline
$\mu$ & $G_{n}(k)$ & $g_{n}(x)$ & $g(x)$ & $H_{n}(k)$ & $h_{n}(x)$ & $h(x)$ \\\hline
$0$ & $\frac{k}{n(n-1)}$ & $\frac{\lfloor nx\rfloor}{n-1}$ & $x$ & $\frac
{k}{k+1}$ & $\frac{n}{\lfloor nx\rfloor+1}$ & $\frac{1}{x}$ \\\hline
\end{tabular}
\end{table}

The corresponding IVP is:
\[
\left\{
\begin{array}
[c]{c}%
y^{\prime}=\dfrac{y}{x}-x\\
\multicolumn{1}{l}{y(1)=0}%
\end{array}
\right.
\]
with solution:
\[
f(x)=x-x^{2}%
\]

In this example, as in most of the subsequent ones, conditions (1), (2),
(3i), and (3ii) from Theorem \ref{main} are again straightforward (in fact we will not mention them any
more). Conditions (3iii) and (3iv) hold because:
\begin{equation*}
  V_n(k)=\frac{k-n+1}{(n-1) n},\,\,\,
  M_{n}(k)=2
\end{equation*}

By Theorem \ref{main}, the sequence $F_{n}(\lfloor
nx\rfloor)$ converges uniformly to $f(x)$ in $[0,1]$. Since
$f(x)$ reaches its maximum at $x=\frac12$ and $f(\frac12)=\frac14$ we can apply
Proposition \ref{conv} to get the well-known results \cite{mejorpeor,Rose,posdoc2021}:
\begin{align*}
\lim_{n} \frac{\kappa_{n}}{n}=\frac12,\,\,\,
\lim_{n} \mathbf{P}_{n}=\frac14.
\end{align*}

\subsection{The Best-or-Worst variant}
In this version, the aim is to select either the best or the worst candidate, and it is also an optimal stopping problem. The
corresponding probabilities $F_{n}(k)$  of successfully choosing the best or worst candidate using a threshold
strategy with threshold $k$ satisfy \cite{mejorpeor}:
\begin{align*}
F_{n}(k)  &  =\frac{2}{n}+ \frac{k-1}{k+1}F_{n}(k+1),\\
F_{n}(n)  &  =0.
\end{align*}

The relevant data are given in Table 2:
\begin{table}[h]
\caption{Data for the Best-or-worst variant.}%
\label{Tbow}%
\begin{tabular}
[c]{|c|c|c|c|c|c|c|}\hline
$\mu$ & $G_{n}(k)$ & $g_{n}(x)$ & $g(x)$ & $H_{n}(k)$ & $h_{n}(x)$ & $h(x)$ \\\hline
$0$ & $\frac{2}{n}$ & $2$ & $2$ & $\frac{k-1}{k+1}$ & $\frac{2n}{\lfloor
nx\rfloor+1}$ & $\frac{2}{x}$ \\\hline
\end{tabular}
\end{table}

The corresponding IVP is:
\[
\left\{
\begin{array}
[c]{c}%
y^{\prime}=\dfrac{2y}{x}-2\\
\multicolumn{1}{l}{y(1)=0}%
\end{array}
\right.
\]
whose solution is:
\[
f(x)=2x-2x^{2}%
\]

Conditions (3iii) and (3iv) in Theorem \ref{main} hold in this case because $V_n(k)=0$ and $M_n(k)=4$. As a consequence, $F_{n}(\lfloor
nx\rfloor)$ converges uniformly to $f(x)$ in $[0,1]$. Since
$f(x)$ reaches its maximum at $x=\frac12$ and $f(\frac12)=\frac12$
Proposition \ref{conv} gives the results from \cite{mejorpeor}:
\begin{equation*}
\lim_{n} \frac{\kappa_{n}}{n}=\frac12,\,\,\,
\lim_{n} \mathbf{P}_{n}=\frac12
\end{equation*}

\subsection{The Secretary Problem with Uncertain Employment}

This variant \cite{refusal} introduces the possibility that each candidate
can be effectively hired only with certain fixed probability
$0<p\leq1$ (independent of the candidate). If a specific candidate
cannot be hired, it cannot be chosen and the process must
continue. Obviously, the case $p=1$ is the classical problem while the case
$p=0$ is absurd.

In this situation, $p^{(n)}_{k}=\dfrac{p}{k}$ and
$P^{(n)}_k(1)=\dfrac{k}{n}$. Hence, the probabilities $F_{n}(k)$
satisfy the following recurrence relation.
\begin{align*}
F_{n}(k)  &  =\frac{p}{n}+ \left(  1-\frac{p}{k+1}\right)  F_{n}(k+1),\\
F_{n}(n)  &  =0.
\end{align*}

Table \ref{Tuncer} summarizes the relevant data (all the computations are straightforward).

\begin{table}[h]
\caption{Data for the Secretary Problem with uncertain employment.}%
\label{Tuncer}%
\begin{tabular}
[c]{|c|c|c|c|c|c|c|}\hline
$\mu$ & $G_{n}(k)$ & $g_{n}(x)$ & $g(x)$ & $H_{n}(k)$ & $h_{n}(x)$ & $h(x)$ \\\hline
$0$ & $\frac{p}{n}$ & $p$ & $p$ & $\frac{k+1-p}{k+1}$ & $\frac{pn}{\lfloor
nx\rfloor+1}$ & $\frac{p}{x}$ \\\hline
\end{tabular}
\end{table}

The corresponding IVP is:
\[
\left\{
\begin{array}
[c]{c}%
y^{\prime}=\dfrac{py}{x}-p\\
\multicolumn{1}{l}{y(1)=0}%
\end{array}
\right.
\]
with solution:
\[
f(x)=\frac{p(x^{p}-x)}{1-p}%
\]

Conditions (3iii) and (3iv) of Theorem \ref{main} hold because
$V_n(k)=0$ and $M_n(k)=   p^{2}\left(  \frac{n}{k}\right)
^{2-p}$. Thus, $F_{n}(\lfloor
nx\rfloor)$ converges uniformly to $f(x)$ in $[0,1]$. The function
$f(x)$ reaches its maximum at $x=p^{\frac{1}{1-p}}$, and $f(p^{\frac{1}{1-p}})=p^{\frac{1}{1-p}}$, so that
Proposition \ref{conv} provides the results from \cite{refusal}:
\begin{align*}
\lim_{n} \frac{\kappa_{n}}{n}=p^{\frac{1}{1-p}},\,\,\,
\lim_{n} \mathbf{P}_{n}=p^{\frac{1}{1-p}}%
\end{align*}
Observe that, as expected, if $p\to 1$ these values converge to the solution of the classical problem.

\subsection{The Secretary Problem with interview cost}

In this variant \cite{coste1}, a cost $\frac{c}{n}$ (with $0\leq c<1$) for each observed
candidate is introduced (if $c=0$ the problem is the classical one). The difference with the classical problem (cf. subsection 3.1) is
that, in this situation, $p^{(n)}_{k}=\frac{1}{k}$, the payoff
function is $P^{(n)}_{k}(1)=\frac{k}{n}(1-c)$ and
$\mu=-c$. Thus,
\begin{align*}
F_{n}(k)  &  =\frac{1-c}{n}+ \frac{k}{k+1}F_{n}(k+1),\\
F_{n}(n)  &  =-c.
\end{align*}

Table \ref{Tcost} contains the relevant data.

\begin{table}[h]
\caption{Data for the Secretary Problem with interview cost.}%
\label{Tcost}
\begin{tabular}
[c]{|c|c|c|c|c|c|c|}\hline
$\mu$ & $G_{n}(k)$ & $g_{n}(x)$ & $g(x)$ & $H_{n}(k)$ & $h_{n}(x)$ & $h(x)$ \\\hline
$-c$ & $\frac{1-c}{n}$ & $1-c$ & $1-c$ & $\frac{k}{k+1}$ & $\frac{n}{\lfloor
nx\rfloor+1}$ & $\frac{1}{x}$ \\\hline
\end{tabular}
\end{table}

The corresponding IVP is:
\[
\left\{
\begin{array}
[c]{c}%
y^{\prime}=\dfrac{y}{x}-(1-c)\\
\multicolumn{1}{l}{y(1)=-c}%
\end{array}
\right.
\]
with solution:
\[
f(x)=-cx+cx\log x-x\log x
\]

In this case, $V_n(k)=0$ and $M_n(k)=\frac{(1-c)n}{k}$. Theorem \ref{main} holds and $F_{n}(\lfloor
nx\rfloor)$ converges uniformly to $f(x)$ in $[0,1]$. Since
$f(x)$ reaches its maximum at $x=e^{\frac{1}{c-1}}$ and $f\left(e^{\frac{1}{c-1}}\right)=(1-c)e^{\frac{1}{c-1}}$, Proposition \ref{conv} gives the results from \cite{coste1}:

\begin{equation*}
\lim_{n}\frac{\kappa_{n}}{n}=e^{\frac{1}{c-1}},\,\,\,
\lim_{n}\mathbf{P}_{n}=(1-c)e^{\frac{1}{c-1}}%
\end{equation*}
For $c=0$ we obviously recover the values for the classical problem.

\subsection{The win-lose-or-draw Secretary Problem}

In this variant, there is a reward $\alpha$ when choosing the best candidate, a penalty $\beta$ when choosing a wrong one, and a different penalty $\gamma$ when choosing none. The original version of this variant \cite{ferguson} has $\alpha=\beta=1$, and $\gamma=0$.

This problem has $p^{(n)}_{k}=\frac{1}{k}$, and the payoff function is
$$
P^{(n)}_{k}(1)=\alpha\frac{k}{n}-\beta\left(  1-\frac{k}{n}\right)
$$
so that the $F_{n}(k)$ are defined recursively as
\begin{align*}
F_{n}(k)  &  =\frac{(\alpha+\beta)(k+1)-\beta n}{(k+1)n}+ \frac{k}{k+1}%
F_{n}(k+1),\\
F_{n}(n)  &  =-\gamma.
\end{align*}
Notice that if $\alpha=1-\gamma$ and $\beta=0$, we are in the previous case with $c=\gamma$. Also, if $\alpha=1$, and $\beta=\gamma=0$ we are in the Classical Secretary Problem.

The relevant data are contained in Table \ref{Twlod}.
\begin{table}[h]
\caption{Data for the Win-lose-or-draw Secretary Problem.}%
\label{Twlod}
\begin{tabular}
[c]{|c|c|c|c|c|c|c|}\hline
$\mu$ & $G_{n}(k)$ & $g_{n}(x)$ & $g(x)$ & $H_{n}(k)$ & $h_{n}(x)$ & $h(x)$ \\\hline
$-\gamma$ & $\frac{(\alpha+\beta)(k+1)-\beta n}{(k+1)n}$ & $\frac
{(\alpha+\beta)(\lfloor nx\rfloor+1)-\beta n}{\lfloor nx\rfloor+1}$ &
$\alpha+\beta-\frac{\beta}{x}$ & $\frac{k}{k+1}$ & $\frac{n}{\lfloor
nx\rfloor+1}$ & $\frac{1}{x}$ \\\hline %
\end{tabular}
\end{table}

The corresponding IVP is:
$$
\left\{
\begin{array}
[c]{c}%
y^{\prime}=\displaystyle{\frac{y}{x}}-(\alpha+\beta-\frac{\beta}{x})\\
\multicolumn{1}{l}{y(1)=-\gamma}%
\end{array}
\right.
$$
whose solution is:
\[
f(x)=-(\alpha+\beta)x\log x+\beta(x-1)-\gamma x.
\]

Theorem \ref{main} holds because $V_n(k)=0$ and $M_n(k)=\frac{(\alpha+\beta)n}{k}$. As a consequence, $F_{n}(\lfloor
nx\rfloor)$ converges uniformly to $f(x)$ in $[0,1]$. Since
$f(x)$ reaches its maximum at $x=e^{\frac{-\alpha-\gamma}{\alpha+\beta}}$,
Proposition \ref{conv} gives
\begin{equation*}
  \lim_{n}\frac{\kappa_{n}}{n}=e^{\frac{-\alpha-\gamma}{\alpha+\beta}},
  \,\,\,
  \lim_{n}\mathbf{P}_{n}=f\left(  e^{\frac{-\alpha-\gamma}{\alpha+\beta}%
}\right)
\end{equation*}

For $\alpha=\beta=1$ and $\gamma=0$ we get the results given in \cite{ferguson}:
\begin{align*}
  \lim_{n} \frac{\kappa_{n}}{n}=\frac{1}{\sqrt{e}}=0.60653\dots,
  \,\,\,
 \lim_{n} \mathbf{P}_{n} =\frac{2}{\sqrt{e}}-1 = 0.213061\dots
\end{align*}

\subsection{The Best Choice Duration Problem}
This variant specifies a reward of $\frac{n+1-k}%
{n}$ when choosing the best candidate at step $k$ (notice that the
reward decreases with $k$), so that there is an incentive to make the correct choice as soon as possible. We refer to \cite{FER2} and \cite{dura2} for previous studies on this problem.

Setting $p^{(n)}_{k}=\frac{1}{k}$, the payoff function $P^{(n)}_{k}$ is
$$
P^{(n)}_{k}(1)=\frac{k(n+1-k)}{n^{2}}.
$$
so that $F_n(k)$ is given by:
\begin{align*}
F_{n}(k)  &  =\frac{n-k}{n^{2}}+ \frac{k}{k+1}F_{n}(k+1),\\
F_{n}(n)  &  =0.
\end{align*}

Table \ref{Tbcd} includes the summary of the relevant information.
\begin{table}[h]
\caption{Data for the Best Choice Duration Problem.}%
\label{Tbcd}
\begin{tabular}
[c]{|c|c|c|c|c|c|c|}\hline
$\mu$ & $G_{n}(k)$ & $g_{n}(x)$ & $g(x)$ & $H_{n}(k)$ & $h_{n}(x)$ & $h(x)$ \\\hline
$0$ & $\frac{n-k}{n^{2}}$ & $\frac{n-\lfloor nx\rfloor}{n}$ & $1-x$ &
$\frac{k}{k+1}$ & $\frac{n}{\lfloor nx\rfloor+1}$ & $\frac{1}{x}$ \\\hline
\end{tabular}
\end{table}

The IVP for this variant is:
$$
\left\{
\begin{array}
[c]{c}%
y^{\prime}=\displaystyle{\frac{y}{x}}-(1-x)\\
\multicolumn{1}{l}{y(1)=0}%
\end{array}
\right.
$$
with solution:
\[
f(x)=x^{2}-x-x\log x.
\]

In this case, $V_n(k)=\frac{1}{n}$ and $M_n(k)\leq 2+\frac{n}{k}$, so that all the hypotheses from Theorem \ref{main} hold. Thus,
$F_{n}(\lfloor
nx\rfloor)$ converges uniformly to $f(x)$ in $[0,1]$. The maximum of
$f(x)$ is reached at $x=\vartheta=-\frac{1}{2}W(-2e^{-2})$ with $f(\vartheta)=\vartheta-\vartheta^{2}$.
Proposition \ref{conv} the gives the known results \cite{FER2,dura2}:

\begin{equation*}
\lim_{n}\frac{\kappa_{n}}{n}=\vartheta=0.2031878\dots, \,\,\, \lim_{n}\mathbf{P}_{n}=f(\vartheta)=0.161902559\dots
\end{equation*}

\subsection{A simplified Multicriteria Secretary Problem}
In this case, the $n$ candidates are ranked across $m\geq 1$ independent
attributes ($m=1$ is the just classical case), and the aim is to choose a candidate which is the best in one of the
attributes. When a candidate is chosen, it is specified in which attribute
it is considered to be the best. This is a simplification of the original variant \cite{multiatri}, in which the attribute does not have to be specified. This simplification can be seen to be asymptotically negligible, but we do not get into details.

In this case, $p^{(n)}_{k}=  1-\left(  \frac{k-1}{k}\right)  ^{m}$, the payoff function is $P^{(n)}_{k}(1)=\frac{k}{n}$ and $F_n(k)$ is given by:
\begin{align*}
F_{n}(k)  &  =\left(  1-\left(  \frac{k}{k+1}\right)  ^{m}\right)
\frac{k+1}{n}+\left(  \frac{k}{k+1}\right)  ^{m}F_{n}(k+1),\\
F_{n}(n)  &  =0.
\end{align*}

The relevant information is summarized in Table \ref{Mult}.
\begin{table}[h!]
\caption{Data for the Multicriteria Secretary Problem.}%
\label{Mult}
\begin{tabular}
[c]{|c|c|c|c|c|c|c|}\hline
$\mu$ & $G_{n}(k)$ & $g_{n}(x)$ & $g(x)$ & $H_{n}(k)$ & $h_{n}(x)$ & $h(x)$ \\\hline
$0$ & $\left(  1-\left(  \frac{k}{k+1}\right)  ^{m}\right)  \frac{k+1}{n}$ &
$(\lfloor nx\rfloor+1)\left(  1-\left(  \frac{\lfloor nx\rfloor}{\lfloor
nx\rfloor+1}\right)  ^{m}\right)  $ & $m$ & $\left(  \frac{k}{k+1}\right)
^{m}$ & $n\left(  1-\left(  \frac{\lfloor nx\rfloor}{\lfloor nx\rfloor
+1}\right)  ^{m}\right)  $ & $\frac{m}{x}$
\\\hline
\end{tabular}
\end{table}

The corresponding IVP is:
\[
\left\{
\begin{array}
[c]{c}%
y^{\prime}=\displaystyle{\frac{m y}{x}}-m\\
y(1)=0\text{ \ \ \ \ \ \ }%
\end{array}
\right.
\]

whose solution for $m>1$ is (the case $m=1$ should be addressed separately, but it is just the classical case):
\[
f(x)=-\frac{m\left(  x^{m}-x\right)  }{m-1}%
\]

In this problem,
\[
V_{n}(k)=\frac{\left(  k^{m}(k+1)^{1-m}-k+m-1\right)  \left(
-mn\left(  \frac{k+1}{n}\right)  ^{m}+k+1\right)  }{(k+1)(m-1)} \]
so it holds that $|V_{n}(k)|<m/k$, whereas
\[
\left\vert f^{\prime\prime}(x)\right\vert =m^{2}x^{m-2}\leq m^{2}
\]
which give conditions (3iii) and (3iv) of Theorem \ref{main}. Thus, $F_{n}(\lfloor
nx\rfloor)$ converges uniformly to $f(x)$ in $[0,1]$. The function
$f(x)$ reaches its maximum at $x=m^{\frac{1}{1-m}}$ and $f\left(m^{\frac{1}{1-m}}\right)=m^{\frac{1}{1-m}}$, so
Proposition \ref{conv} gives the results from \cite{multiatri}:
\begin{equation*}
\lim_{n}\frac{\kappa_{n}}{n}=\displaystyle{m^{\frac{1}{1-m}}},\,\,\,
\lim_{n}\mathbf{P}_{n}  = \displaystyle{m^{\frac{1}{1-m}}}%
\end{equation*}

\subsection{The Secretary Problem with a random number of applicants}
We now depart slightly from the classical setting by letting $N$ (the number of candidates) be a random variable uniform over $\left\{ 1,\ldots, n \right\}$, as in \cite{ferguson,sonin,ras}.

First, let $\mathfrak{M}_{n}(k)$ be the probability that, when rejecting a candidate in the $k$-th interview, there are still more available candidates. Also, let $P^{A}_{n}(k)$ be the probability of success when choosing, in the $k$-th interview, a candidate which is better than all the previous ones. Then, the following equalities hold:
\begin{itemize}
\item $\mathfrak{M}_{n}(0)=1$, and  for $k>0$:
\[
\mathfrak{M}_{n}(k)=\frac{n-k}{n-k+1}%
\]
\item Using the well-known digamma function $\psi$,
\[
P^{A}_{n}(k)=\frac{1}{n-k+1}\sum_{i=k}^{n}\frac{k}{i}=\frac{k (\psi
(n+1)-\psi(k))}{n-k+1}%
\]
\end{itemize}

On the other hand, let $F_{n}(k)$ be the probability of success when rejecting the $k$-th candidate and choosing, later on, the one which is better than all the previous ones. That is, the probability of success using the threshold strategy $k$ assuming that there are at least $k$ candidates. The following recurrence relations hold:
\begin{align*}
F_{n}(k)  &  =\mathfrak{M}_{n}(k)\frac{1}{k+1}P_{n}^{A}(k+1)+\mathfrak{M}%
_{n}(k)\frac{k}{k+1}F_{n}(k+1),\\
F_{n}(n)  &  =0.
\end{align*}

Finally, the prior probability of there being at least $k$ candidates (or what is
the same, the probability that the $k$-th interview can be reached) is
$L_n(k)=\frac{n-k+1}{n}$. As a consequence, the probability of
success using the threshold $k$ is given by
$$
P_n(k)= L_n(k) F_{n}(k)
$$
Obviously, $L_n(\lfloor n x\rfloor)$ converges
uniformly to the function $1-x$ in the interval $[0,1]$, so we just need to study the uniform convergence of
$F_n(\lfloor nx\rfloor)$.

To do so, the relevant data is summarized in Table 8.

\begin{table}[h]
\caption{Data for Random Secretary Problem.}%
\label{Mulrandom}
\begin{tabular}
[c]{|c|c|c|c|c|c|c|}\hline
$\mu$ & $G_{n}(k)$ & $g_{n}(x)$ & $g(x)$ & $H_{n}(k)$ & $h_{n}(x)$ & $h(x)$ \\\hline
$0$ & $\frac{\mathfrak{M}_{n}(k)P_{n}^{A}(k+1)}{k+1}$ & $\frac{\mathfrak{M}_{n}(\lfloor
nx\rfloor)P_{n}^{A}(\lfloor nx\rfloor+1)}{\lfloor nx\rfloor+1}$ & $\frac
{\log(x)}{x-1}$ & $\frac{\mathfrak{M}_{n}(k)k}{k+1}$ & $ n\left(1-\frac{\mathfrak{M}_{n}(\lfloor nx\rfloor)\lfloor nx\rfloor}{\lfloor nx\rfloor+1}\right)$ &$\frac{1}{x-x^2}$\\\hline
\end{tabular}
\end{table}

The corresponding IVP is:
\[
\left\{
\begin{array}
[c]{c}%
y^{\prime}=\displaystyle{\frac{y}{x-x^{2}}}-\displaystyle{\frac{\log(x)}{x-1}}\\
y(1)=0\text{ \ \ \ \ \ \ \ \ \ \ \ \ \ \ \ }%
\end{array}
\right.
\]

Note that this differential equation is singular at the initial condition $x=1,y=0$. From a formal point of view, the function
\[
f(x)=-\frac{x \log^{2}(x)}{2 (x-1)},\,\,\, f(0)=f(1)=0.
\]
satisfies the differential equation in $(0,1)$, and is in fact continuous in $[0,1]$. Hence, we need to verify that the conditions of Theorem \ref{main} hold for it. Conditions (3i) and (3ii) are obvious. Regarding condition (3iii) we observe that
$$
V_n(k)=\frac{k n (n-k) \left(H_n-H_{k-1}\right)}{(k+1)
   (-k+n+1)^2}-\frac{n (k+n+1) \log
   ^2\left(\frac{k+1}{n}\right)}{2 (k-n-1)
   (k-n+1)^2}-\frac{n \log
   \left(\frac{k+1}{n}\right)}{k-n+1}< \frac{1}{k}
$$
while for condition (3iv), from
\[
f^{\prime\prime}(x)=\frac{(x-\log(x)-1) (-x+x \log(x)+1)}{(x-1)^{3} x}%
\]
it follows that:
\[
M_{n}(k)=\left|  f^{\prime\prime}\left(  \frac{k+1}{n}\right)  \right|
\leq\frac{n}{ k}.
\]
As a consequence, $F_{n}(\lfloor nx\rfloor)$ converges uniformly to $f(x)$
on $[0,1]$ and, uniformly in $[0,1]$, we have that
\[
\lim_{n}P_{n}(\lfloor nx\rfloor)=\lim
_{n}L_{n}(\lfloor nx\rfloor) \lim_{n} F_{n}(\lfloor nx\rfloor)=(1-x)f(x)=\frac
{x\log^{2}(x)}{2}%
\]
Moreover, the maximum of this function in $[0,1]$ is reached at $x=e^{-2}$, so Proposition \ref{conv} gives the know results from \cite{sonin,ras}

\begin{equation*}
  \lim_{n}\frac{\kappa_{n}}{n}  =e^{-2}=0.1353352...,
  \,\,\,
\lim_{n}\mathbf{P}_{n} =\mathbf{P}(e^{-2})=2e^{-2}=0.27067056...
\end{equation*}

\section{Four original examples}

We now devise four original examples in which our technique works straightforwardly. The first one is a lottery in which the winning payoff increases at each stage, but which may end up with no prize at all. The three remaining ones are new versions of the Secretary Problem not considered in the literature so far.

\subsection{Lotteries with increasing winning payoff}
There are $n$ balls in an urn, only one of which is white. The game has $n$ identical stages in which a ball is randomly drawn from the urn and a decision is taken:
\begin{itemize}
\item If the ball is black, it is returned and the player proceeds to the next stage.
\item If the ball is white at the $k$-th stage, the player can choose between ending the game with a payoff $Y(k/n)$ (where $Y(x)$ is a function defined in $[0,1]$), or returning it to the urn and proceed to the next stage.
 \item The game ends at the end of the $n$-th stage.
\end{itemize}

Let $P_{n}^{R}(k)$ be the expectation of winning after ending the $k$-th stage, when following  the optimal strategy. As we mentioned in the Introduction, whatever this strategy is, the expectation of winning following it is $P^R_n(0)$. The functions $P_{n}^{R}(k)$ satisfy the recurrence:%
\begin{align*}
P_{n}^{R}(k)&=\frac{1}{n}\max\left\{  Y\left(  \frac{k+1}{n}\right)  ,P_{n}%
^{R}(k+1)\right\}  +\frac{n-1}{n}P_{n}^{R}(k+1),\\
P_{n}^{R}(n)&=0.
\end{align*}

If the payoff function $Y(x)$ is non-decreasing, it can be easily seen that the optimal strategy is threshold. In is described in the following proposition.

\begin{prop}
In the previous setting let us assume that the payoff function $Y(x)$ is non-decreasing. Then, for all $n$, there exists $\kappa_n$ such that the optimal strategy consists in stopping whenever a white ball appears after the $\kappa_n$-th stage and rejecting it before that stage.
\end{prop}

Now, let $F_{n}(k)$ be the expected payoff when using a threshold strategy of threshold $k$. These functions satisfy the recurrence relation:
\begin{align*}
F_{n}(k)  &  =\frac{1}{n}Y\left(  \frac{k+1}{n}\right)  +\frac{n-1}{n}%
F_{n}(k+1),\\
F_{n}(n)  &  =0.
\end{align*}

The relevant data for this game is summarized in Table \ref{juegoI}.
\begin{table}[h]
\caption{Data for the lottery with increasing payoff.}%
\label{juegoI}
\begin{tabular}
[c]{|c|c|c|c|c|c|c|c|c|}\hline
$\mu$ & $G_{n}(k)$ & $g_{n}(x)$ & $g(x)$ & $H_{n}(k)$ & $h_{n}(x)$ & $h(x)$ \\\hline
$0$ & $\frac{1}{n}Y\left(  \frac{k+1}{n}\right)  $ & $Y\left(  \frac{\lfloor
nx\rfloor+1}{n}\right)  $ & $Y\left(  x\right)  $ & $\frac{n-1}{n}$ & $1$ &
$1$  \\\hline
\end{tabular}
\end{table}

Consequently, we must solve the IVP
\[
\left\{
\begin{array}
[c]{c}%
y^{\prime}=y-Y(x)\\
\multicolumn{1}{l}{y(1)=0}%
\end{array}
\right.
\]
Assuming that $Y(x)$ is Lipschitz in $[0,1]$, its solution is given by
\[
f(x)=e^{x}\int_{x}^{1}e^{-u}Y(u)\,du
\]
In order to apply Theorem \ref{main}, note that condition (3iii) holds because
$$\frac{V_{n}(k)}{n}=\frac{Y\left(  \frac{k}{n}\right)  -Y\left(  \frac{k+1}%
{n}\right)  }{n}$$
so that, $Y(x)$ being Lipschitz, it follows that
$$\sum_{k=1}^{n-2} \frac{V_{n}(k)}{n} = \frac{Y\left(  \frac{1} {n}\right)- Y\left(  \frac{n-1} {n}\right)}{n} \longrightarrow 0.$$
Also, condition (3iv) is satisfied because $f''$ is bounded in $[0,1]$, since:
\[
f''(x)=f(x)-Y(x)-Y'(x)
\]

Thus, due to Theorem \ref{main} $F_{n}(\lfloor nx\rfloor)$ converges uniformly to $f(x)$ in $[0,1]$. Moreover, if $\vartheta$ is the unique solution of $f(x)=Y(x)$ we have that $f'(\vartheta)=0$, $f''(\vartheta)=-Y'(\vartheta)<0$, and by Proposition \ref{conv}:
\begin{align*}
&  \lim_{n}\frac{\kappa_{n}}{n}=\vartheta\\
&  \lim_{n}\mathbf{P}_{n}=Y(\vartheta)
\end{align*}

\begin{ejem}
Let us consider the payoff function $Y(x)=x$. Then, it follows that $f(x)=x-2 e^{x-1}+1$. If we set $n=10^7$ it can be directly computed using the dynamic program that $\kappa_n=3068528$, and $\mathbf{P}_{n}= 0.3068528540974\dots$.

Now, in this case, and according to our previous discussion $\lim_{n}\frac{\kappa_{n}}{n}=\vartheta=\lim_{n}\mathbf{P}_{n}$ where $\vartheta=1-\log 2=0.30685281944005\dots$ is the unique solution to $x-2^{x-1}+1=x$.
\end{ejem}

\subsection{Secretary problem with a wildcard}
There are $n+1$ balls in an urn: $n$ of them are ranked from $1$ to $n$, and the other one is a wildcard. At each stage of the game, a ball is extracted. The rank of each ball is known only when it is extracted. The player decides according to the following scheme:
\begin{itemize}
\item If the ball is the wildcard, he can stop the game and get a payoff of $1/2$, or he can decide to continue the game discarding the wildcard (i.e. it is not returned to the urn).
\item Otherwise, the player can either stop the game, in which case he wins $1$ if the ball is the best, and $0$ otherwise; or he can discard the ball and continue the process.
\end{itemize}
Thus, once the wildcard is rejected, the game goes on according to the rules of the classical secretary problem.

Let $E_n(k)$ be the expected payoff when rejecting the $k$-th ball if the wildcard has not appeared in the $k-1$ previous extractions. This $E_n(k)$ satisfies the following recurrence (dynamic program), where, as usual, $\textbf{P}_n=E_n(0)$ is the expected payoff using the optimal strategy.
\begin{small}
\begin{align*}
E_n(k)   = &\frac{1/2}{n-k+1}+\frac{n-k}{n-k+1}\cdot\frac{1}{k+1}\cdot\max
             \left\{\frac{k+1}{n},E(k+1)\right\}+\\
  &+ \frac{n-k}{n-k+1}\cdot\frac{k}{k+1}\cdot E_n(k+1);\\
  E_n(n)  =& \frac{1}{2}
\end{align*}
\end{small}

The optimal strategy in this game is a threshold strategy, as we see in the following result.

\begin{prop}
  \label{ju} For each $n>1$ there is $\kappa_n$ such that the following strategy is optimal:
  \begin{enumerate}
  \item Stop the game whenever the wildcard is extracted. Otherwise,
  \item Before the $\kappa_n$-th extraction continue the game, and
  \item From the $\kappa_n$-th extraction on, choose any ball which is better than the previous ones (or is the wildcard, obviously).
  \end{enumerate}
\end{prop}
\begin{proof}
Certainly, if the wildcard in encountered, it must always be chosen because if it is discarded we are in the classical secretary problem in wich the expected payoff is always smaller than $1/2$ and there is no value in continuing with the process.

On the other hand, the function $E_n(k)$ is trivially non-increasing in $k$. This implies that if, for a specific $k$, the optimal decision is to stop with any ball better than the previous ones, then the same holds for all values greater than $k$. In other words,
$$E_n(k)\leq k/n \Longrightarrow E_n(k+1)\leq  \frac{k+1}{n}$$
and this finishes the proof.
\end{proof}

Let now $F_n(k)$ be the expected payoff following a strategy that consists in rejecting the $k$-th ball and then choosing either the wildcard or the first ball which is better than the previous ones. Thus,
\begin{align*}
F_{n}(k)  &  =
\frac{3 n-2k}{2 n (n-k+1)}+\frac{k (n-k)}{(k+1)
   (n-k+1)}F_{n}(k+1),\\
F_{n}(n)  &  =1/2.
\end{align*}

\begin{table}[h]
\caption{Data for the wildcard game.}\label{juegoII}%
\begin{tabular}
[c]{|c|c|c|c|c|c|c|c|c|}\hline
$\mu$ & $G_{n}(k)$ & $g_{n}(x)$ & $g(x)$ & $H_{n}(k)$ & $h_{n}(x)$ & $h(x)$ \\\hline
$\frac{1}{2}$ & $\frac{3n-2k}{2n(n-k+1)}$ & $\frac{3n-2\lfloor nx\rfloor
}{2(n-\lfloor nx\rfloor+1)}$ & $\frac{3-2x}{2-2x}$ & $\frac{(n-k)k}%
{(n-k+1)(k+1)}$ & $\frac{n(n+1)}{(n-\lfloor nx\rfloor+1)(\lfloor nx\rfloor
+1)}$ & $\frac{1}{x-x^{2}}$  \\\hline
\end{tabular}
\end{table}

Table \ref{juegoII} contains the relevant data and the IVP is:
\[
\left\{
\begin{array}
[c]{c}%
y'=\displaystyle\frac{y}{x-x^{2}}-\frac{3-2x}{2-2x}\\[1em]
\multicolumn{1}{l}{y(1)=1/2}%
\end{array}
\right.
\]
which, despite the singularity at $x=1$, has the unique solution (continuous in $[0,1]$):
\[
f(x)=\frac{-2x^{2}+2x+3x\log(x)}{2(x-1)}%
\]
Condition (3iii) of Theorem \ref{main} holds because
$$
V_n(k)= \frac{3 n \left((k+n+1) \log \left(\frac{k+1}{n}\right)-2
    (k-n+1)\right)}{2 (k-n-1) (k-n+1)^2}<\frac{1}{k}
$$
while condition (3iv) also holds because:
$$
\left|f''(x)\right|
=  \left|-\frac{3 \left(x^2-2 x \log (x)-1\right)}{2 (x-1)^3 x}\right|
$$
is an decreasing function and
$$
  |M_n(k)|= - f''(k/n)=\frac{3 n^2 \left(-k^2+2 k n \log
      \left(\frac{k}{n}\right)+n^2\right)}{2 k (k-n)^3}<\frac{n}{k}
$$
Hence, we conclude that $F_{n}(\lfloor nx\rfloor)$ converges uniformly to $f(x)$ in $[0,1]$ and we have the following

\begin{prop}
  $$
  \lim_n \frac{\kappa_n}{n} =-\frac{3}{4} W\left(  -\frac{4}{3 e^{4/3}%
    }\right)  = 0.545605016560\dots
  $$
\end{prop}
\begin{proof}
First of all, note that
  $$
  F_n(\kappa_n)>\frac{\kappa_n}{n}\textrm{ and } F_n(\kappa_n+i)\leq\frac{\kappa_n+i}{n}\textrm{ for all } i=1, ..., n-\kappa_n.
  $$
Consequently, the result follows from Proposition \ref{nuevaa}, and the fact that $f(x)=x$ has a single solution in $(0,1]$.
\end{proof}

We also have

\begin{prop}
Let $\vartheta=-\frac{3}{4} W\left(  -\frac{4}{3 e^{4/3}}\right)$. Then,
 $$
\lim_n \mathbf{P}_{n} =\frac{1}{2} \vartheta+ (1- \vartheta) \vartheta
=0.5207226907\dots
$$
\end{prop}
\begin{proof}
The probability of reaching step $\kappa_n$ without having extracted the wildcard is clearly
$1-\frac{k_{n}}{n+1}$. As a consequence,
 \[
\mathbf{P}_{n}=\frac{1}{2} \frac{\kappa_{n}}{n+1}+ \left(1-\frac{\kappa_{n}}{n+1}\right)F_n(k_{n}).
\]
Then, the result follows because $\lim_n F_n(\kappa_n) =f(\vartheta)=\vartheta$, and $\frac{\kappa_n}{n}\longrightarrow \vartheta$ due to the previous proposition.
\end{proof}

\begin{rem}
These results seem to be accurate. In fact, for $n=10^7$ we obtain following the values using directly the dynamic program:
\begin{align*}
\mathbf{P}_{10^7}&=0.520722700032\dots\\
\kappa_{10^7}&=5456050.
\end{align*}
\end{rem}

\subsection{Secretary problem with random interruption}
There are $n$ ranked balls (from $1$ to $n$) in an urn. At each stage
of the game, a ball is extracted. The rank of each ball is known only
when it is extracted. The game is the classical secretary game with
the modification that at each stage, a random event with probability
$1/n$ decides whether the game stops without payoff or continues
(e.gr. the ball may ``blow up'' and end the game with probability
$1/n$).

The probability of success (i.e. choosing the best ball) using the
optimal strategy (whatever this might be) is $\overline{F}_n(0)$, and
can be computed by means of the following dynamic program, where
$\overline{F}_n(k)$ is the probability of success after rejecting the
$k$-th ball and following the optimal strategy from that point on:
\begin{align*}
  \overline{F}_n(k)&=\left(1-\frac{1}{n}\right) \left(\frac{\max
      \left(\frac{k+1}{n},\overline{F}_n(k+1)\right)}{k+1} +
    \frac{k
      \overline{F}_n(k+1)}{k+1}\right)\\
			\overline{F}_n(n)&=0.
\end{align*}

It is easy to see that the optimal strategy is threshold. Let now $F_n(k)$ be the probability of success after rejecting the $k$-th ball, and then choosing the first ball which is better than the all the previous ones. The following recurrence holds:
\begin{align*}
F_n(k)&= \frac{n-1}{n^2}+ \frac{k (n-1)}{(k+1) n}F_n(k+1),\\
F_n(n)&=0.
\end{align*}

The data for this game is summarized in Table \ref{juegoIII}
\begin{table}[h]
\caption{Data for Game III.}%
\label{juegoIII}%
\begin{tabular}
[c]{|c|c|c|c|c|c|c|c|c|}\hline
$\mu$ & $G_{n}(k)$ & $g_{n}(x)$ & $g(x)$ & $H_{n}(k)$ & $h_{n}(x)$ & $h(x)$
 \\\hline
$0$ & $\frac{n-1}{n^2}$ & $ \frac{n-1}{n}$ & $1$ & $\frac{k (n-1)}{(k+1) n}$ & $\frac{\lfloor n x \rfloor +n}{\lfloor n x \rfloor+1}$ & $1+\frac{1}{x}$   \\\hline
\end{tabular}
\end{table}

In this case, the IVP to be solved is
$$
\left\{
\begin{array}
[c]{l}%
  y'=\displaystyle\left(\frac{1}{x}+1\right) y(x)-1\\
  y(1) = 0
\end{array}
\right.
$$

whose solution is
$$
f(x)=e^x x (\text{Ei}(-1)-\text{Ei}(-x)),
$$
where $\text{Ei}(x)$ is the so-called exponential integral function.
$$
\text{Ei}(x)=  \int_{-\infty}^{x}\frac{e^t}{t} dt,
$$
and we extend $f(x)$ to $0$ by continuity as $f(0)=0$.

Condition (3iii)  of Theorem \ref{main} holds because
$$
V_n(k)= \frac{e^{\frac{k+1}{n}}
  \left(\text{Ei}(-1)-\text{Ei}\left(-\frac{k+1}{n}\right)\right)-1}{n}< \frac{1}{k}
$$
while condition (3iv) holds because
$$
f''(x)=e^x (x+2) (\text{Ei}(-1)-\text{Ei}(-x))-\frac{x+1}{x}
$$
and
$$
\sum_{k=1}^{n-2} M_n(k)<\frac{\log(n)}{n}.
$$

Hence, we conclude that $F_{n}(\lfloor nx\rfloor)$ converges uniformly to $f(x)$ in $[0,1]$ and we have the following

\begin{prop}
Let  $\vartheta=0.27105459032\dots$ be the only solution in $(0,1]$ of
  $$
  e^{-x} =\int_{-x}^{-1} \frac{e^t}{t}dt.
  $$
Then,
\begin{align*}
\lim_{n}\frac{\kappa_{n}}{n}  &  =\vartheta \\
\lim_{n}\mathbf{P}_{n}  &  =f( \vartheta) e^{-\vartheta}=0.2066994179096392\dots
\end{align*}
\end{prop}
\begin{proof}
  First of all, note that $F_n(\kappa_n)>\frac{\kappa_n}{n}$ and
  $F_n(\kappa_n+i)\leq\frac{\kappa_n+i}{n}$ for all $i>\kappa_n$. So by Proposition \ref{nuevaa}, $\lim_{n}\frac{\kappa_{n}}{n}$ is the
  only positive root of $f(x)=x$, which is $\vartheta$.
	
	Now, in order to succeed
  using the optimal threshold $\kappa_n$, two successive independent
  events must take place:
  \begin{itemize}
  \item[A)] The $\kappa_n$-th extraction takes place and the game does not end because of
    the random event (i.e. the ball does not ``blow-up''). This happens with probability
    $\left(1- \frac{1}{n}\right)^{\kappa_n}$.
  \item[B)] The $\kappa_n$-th  ball is rejected and, after this rejection, the game ends
    successfully following the threshold strategy with threshold $\kappa_n$. This
    happens with probability $F_n(\kappa_n)$.
  \end{itemize}
Consequently,
  $$
  \mathbf{P}_n=\left(1- \frac{1}{n}\right)^{\kappa_n} F_n(\kappa_n)
  $$
 and the result follows because $\lim_n F_n(\kappa_n)\longrightarrow f(\vartheta)$, and
  $\lim_n \left(1-1/n\right)^{\kappa_n} = e^{-\vartheta}$.
\end{proof}

\begin{rem}
This proposition seems to be accurate. In fact, for $n=10^7$ we obtain following the values using directly the dynamic program:
\begin{align*}
  \mathbf{P}_{10^7}&=0.206699425033\dots,\\
  \kappa_{10^7}&=2710546.
\end{align*}
\end{rem}

\subsection{Secretary problem with penalty if the second best is selected}
This is an original variant in which, if the second best candidate is chosen, then a penalty is incurred. Success provides a payoff of $1$, whereas the penalty is $b\geq 0$. The most similar problem is studied by Gusein-Zade in \cite{gus}, where the aim is to choose the best or the second best candidate with respective payoffs $u_1$ and $u_2$, both greater than $0$. However, our case is not covered because $u_2$ would be $-b<0$.

Let $S_n(k)$ be the probability that, a candidate which is the second best up to the $k$-th interview turns out to be the global second best. By definition:
$$
S_n(k)=\frac{\binom{k}{2}}{\binom{n}{2}}=\frac{k^2-k}{n^2-n}
$$

On the other hand, let $\mathfrak{M}_n(k)$ be the expected payoff when choosing at step $k$ the best candidate to date. Then, $\mathfrak{M}_n(k)$ satisfies the following recurrence
\begin{align*}
 \mathfrak{M}_n(k)&=
    \frac{-b}{k+1}S_n(k+1)+\frac{k}{k+1}\mathfrak{M}_n(k+1),\\
\mathfrak{M}_n(n)&=1
\end{align*}
and it can be seen that
$$
\mathfrak{M}_n(k)=\frac{k (b (k-n)+n-1)}{(n-1) n}.
$$

Just like in the classical Secretary problem we have that
$p^{(n)}_{k}=\frac{1}{k}$, but the difference is that in this situation the payoff function is
$P^{(n)}_{k}(1)=\mathfrak{M}_n(k)$. Consequently, if $E_n(0)$ is the expected payoff using the optimal strategy, then the following dynamic program holds:
\begin{align*}
E_n(k)&= \frac{1}{k+1}
    \max \left\{\mathfrak{M}_n(k+1),E_n(k+1)\right\}
    +\frac{k}{k+1} E_n(k+1),\\
 E_n(n)&=0.
\end{align*}

Thus, reasoning as usual, if $F_n(k)$ is the expected payoff when rejecting the $k$-th candidate and
using $k$ as threshold, we have that
\begin{align*}
F_n(k)&=\frac{1}{k+1} \mathfrak{M}_n(k+1)+\frac{k}{k+1}F_n(k+1),\\
F_n(n)&=0.
\end{align*}

Table \ref{Mulrandom} summarizes the relevant data
 \begin{table}[h]
\caption{Penalty if second best is selected}%
\label{Mulrandom}
\begin{tabular}
[c]{|c|c|c|c|c|c|c|}\hline
$\mu$ & $G_{n}(k)$ & $g_{n}(x)$ & $g(x)$ & $H_{n}(k)$ & $h_{n}(x)$ & $h(x)$ \\\hline
$0$ & $\frac{\mathfrak{M}_{n}(k+1)}{k+1}$ & $\frac{\mathfrak{M}_{n}(\lfloor
nx\rfloor) }{\lfloor nx\rfloor+1}$ & $b (x-1)+1$ & $\frac{1}{k+1}$ & $\frac{1}{\lfloor nx\rfloor+1}$ & $\frac{1}{x }$ \\\hline
\end{tabular}
\end{table}

And the IVP is
\[
\left\{
\begin{array}
[c]{l}%
y'(x)=-b (x-1)+\dfrac{y}{x}-1\\
y(1)=0\text{ \ \ \ \ \ \ \ \ \ \ \ \ \ \ \ }%
\end{array}
\right.
\]
whose solution is
$$
f_b(x)= -b x^2+b x+b x \log (x)-x \log (x)
$$
Condition (3iii) of Theorem \ref{main} holds because
$$
V_n(k)=\frac{-b (3 k+2) n+b (k+1)^2+(b-1) n^2+n}{(k+1) (n-1) n}<\frac{b}{k}
$$
and condition (3iv) holds because:
$$
M_n(k)= f''\left( \frac{k+1}{n}\right)=\frac{b n}{k+1}-2 b-\frac{n}{k+1}<\frac{b n}{k}.
$$

Thus, $F_{n}(\lfloor nx\rfloor)$ converges uniformly to $f(x)$ on
$[0,1]$. Moreover, if $\vartheta_b$ is such that $f_b(\vartheta_b)$ is maximum, by
Proposition \ref{conv}, we have that
\begin{align*}
&  \lim_{n}\frac{\kappa_{n}}{n}=\vartheta_b:= \left\{ \begin{array}{lcc}
           \displaystyle{  e^{-1}} &   if  & b =0 \\
					\\ \displaystyle{{\frac{(1-b)}{2b} W\left( {\frac{b}{1-b} \left(2^{1-b} e^{2 b-1}\right)^{\frac{1}{1-b}}}{ }\right)}{ } } &  if  & 0< b   < 1 \\
             \\  \displaystyle{\frac{1}{2}} &  if & b=1  \\
                          \\ \displaystyle{{\frac{(1-b)}{2b} W_{-1} \left( {\frac{b}{1-b} \left(2^{1-b} e^{2 b-1}\right)^{\frac{1}{1-b}}}{ }\right)}{ } } &  if  &   b  > 1
             \end{array}
   \right.\\
&  \lim_{n}\mathbf{P}_{n}=f_b(\vartheta_b)= \left\{ \begin{array}{lcc}
             e^{-1} &   if  & b =0 \\
             \\ \displaystyle{ \frac{1}{4}} &  if & b=1  \\
             \\ \displaystyle{-\frac{\vartheta_b \left(2 (b-1) \log
   \left(-\frac{\vartheta_b}{2 b}\right)+2
   b+\vartheta_b\right)}{4 b} } &  if  & 0 \neq b   \neq 1
             \end{array}
   \right.
\end{align*}

\begin{ejem}
If $b=2$, using the previous results we have that $lim_{n}\frac{\kappa_{n}}{n}=\vartheta_2=-\frac{1}{4} W_{-1}\left(-\frac{4}{e^3}\right)=0.63741732638\dots$ and $\lim_n \mathbf{P}_n = \vartheta_2  (2-2 \vartheta_2 +  \log (\vartheta_2 ))=0.17518436956\dots$

These results seem accurate since, for $n=10^7$, the following values can be computed directly using the dynamic program:
\begin{align*}
\mathbf{P}_{10^7}&=0.175184397659986\dots,\\
\kappa_{10^7}&=6374173.
\end{align*}
\end{ejem}

\section{Future Perspectives}
Our methodology extends to practically any optimal-stopping problem for which the optimal strategy has a single threshold value. When there are several thresholds, there is an important modification in the theoretical background still undeveloped. On a different note, there are sequences of functions defined by recurrences whose associated functions $f_n(x):=F_n (\lfloor n x\rfloor)$ are not uniformly convergent in the closed interval $[0,1]$ but seem to converge pointwise in the open interval $(0,1)$. We provide some insight on these two issues in what follows.

\subsection{Punctual non-uniform convergence}
Under certain conditions, even though $\{f_n\}$ may not converge uniformly in $[0,1]$, it does converge punctually in $(0,1)$ to a $\mathcal{C}^1$ function $f$ satisfying the differential equation from Theorem \ref{main}. This function $f$ may not extend continuously to $1$ or, even if it does, $f(1)$ may not coincide with the final value $\mu$. We hope to find sufficient conditions guaranteeing this punctual convergence of $\{f_n\}$ in $(0,1)$ to such an $f$, and determining what $f(1)$ must be (or what conditions it must satisfy). In this regards we state the following conjectures:

\begin{con}\label{con:one}
  Let $\{F_n\}, \{G_n\}$ and $\{H_n\}$ be sequences of real functions on $\{1,\ldots, n\}$ satisfying
  \begin{align*}
  F_n(k)&=G_n(k)+H_n(k) F_n(k+1),\\
  F_n(n)&=\mu,
  \end{align*}
  and assume that the functions defined in $[0,1]$ by $g_n(x):=n G_n(\lfloor nx \rfloor)$ and $h_n(x):=n(1-H_n(\lfloor nx \rfloor))$ converge pointwise in $(0,1)$ to continuous functions $g$ and $h$, and that the differential equation
  $$
  y'(x)=-g(x) +h(x) y(x)
  $$
  admits a solution $y(x)$ in $(0,1]$ only for the final condition $y(1)=\Theta$. Then $F_n(\lfloor nx \rfloor)$ converges pointwise in $(0,1)$ to a function $f \in \mathcal{C}^1(0,1]$ satisfying
  \begin{align*}
  f'(x)&=-g(x) +h(x) f(x)\ \mathrm{for}\ x\in(0,1)\\
	f(1)&=\Theta.
  \end{align*}
\end{con}

\begin{con}\label{con:two}
  With the same notation as in Conjecture \ref{con:one}, and under the same conditions on $g_n$ and $h_n$, let us assume that the following limit exists:
  $$
  \Theta = \lim_{n \rightarrow \infty} \lim_{k \rightarrow\infty} F_n(n-k).
  $$
  Then $F_n(\lfloor nx \rfloor)$ converges pointwise in $(0,1)$ to a function $f \in \mathcal{C}^1(0,1]$ satisfying
  \begin{align*}
	f'(x)&=-g(x) +h(x) f(x),\\
	f(1)&=\Theta.
  \end{align*}
\end{con}

We provide two examples to illustrate these conjectures and to show that they seem plausible.

\begin{ejem}\label{exmu}
  Consider the following sequences:
  \begin{align*}
  F_n(k)&=G_n(k)+H_n(k) F_n(k+1),\\
	F_n(n)&=\mu,
  \end{align*}
  where
  $$
  G_n(k):=\frac{k}{n^2}+\frac{2 (k+2 n)}{n (-3 k+3 n+2)},
  $$
  and
  $$
  H_n(k):= \frac{3 n-3 k}{-3 k+3 n+2}.
  $$
  We have $G_n(n-1)+\mu H_n(n-1)=\frac{3 (\mu +2)}{5}$, so that condition (2) of Theorem \ref{main} holds if and only if $\mu=3$. All the other conditions hold irrespective of $\mu$. Consider the corresponding differential equation (obtained using our methodology):
  $$
  y^{\prime}(x)=\frac{2 y(x)}{3-3 x}-\frac{-3 x^2+5 x+4}{3-3 x}.
  $$
It has a single solution in $(0,1]$ with final condition $y(1)=3$, namely:
$$
y(x)=\frac{1}{40} \left(-15 x^2+22 x+113\right)
$$
We plot in Figures \ref{fig:mu3}, \ref{fig:mu83} and \ref{fig:mu103} the functions $F_n(\lfloor n x\rfloor )$  whith $ \mu \in\{3,8/3,10/3\}$ for several values of $n$, to illustrate the likely uniform convergence in the first case, and the non-uniformity in the other two. The punctual convergence to $g(x)$ in $(0,1)$ holds regardless of the value of $\mu$.

\begin{figure}[h]
 \includegraphics[width=3in]{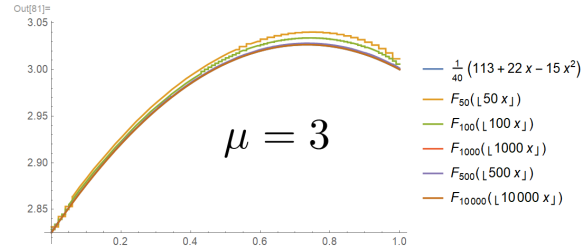}
    \caption{Likely uniform convergence in [0,1] for $\mu=3$ in Example \ref{exmu}.}
		\label{fig:mu3}
\end{figure}

\begin{figure}[h]
\includegraphics[width= 3in]{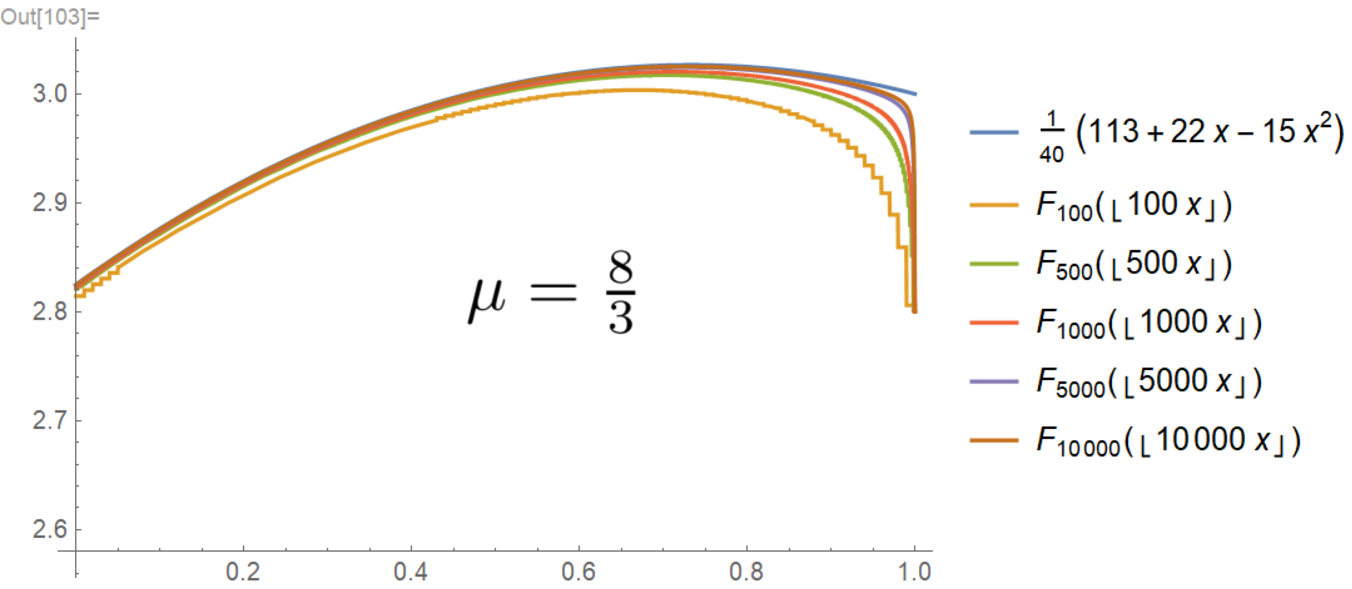}
    \caption{Pointwise, but not uniform, convergence in [0,1) for $\mu=8/3$ in Example \ref{exmu}.}
		\label{fig:mu83}
\end{figure}

\begin{figure}[h]
\includegraphics[width= 3in]{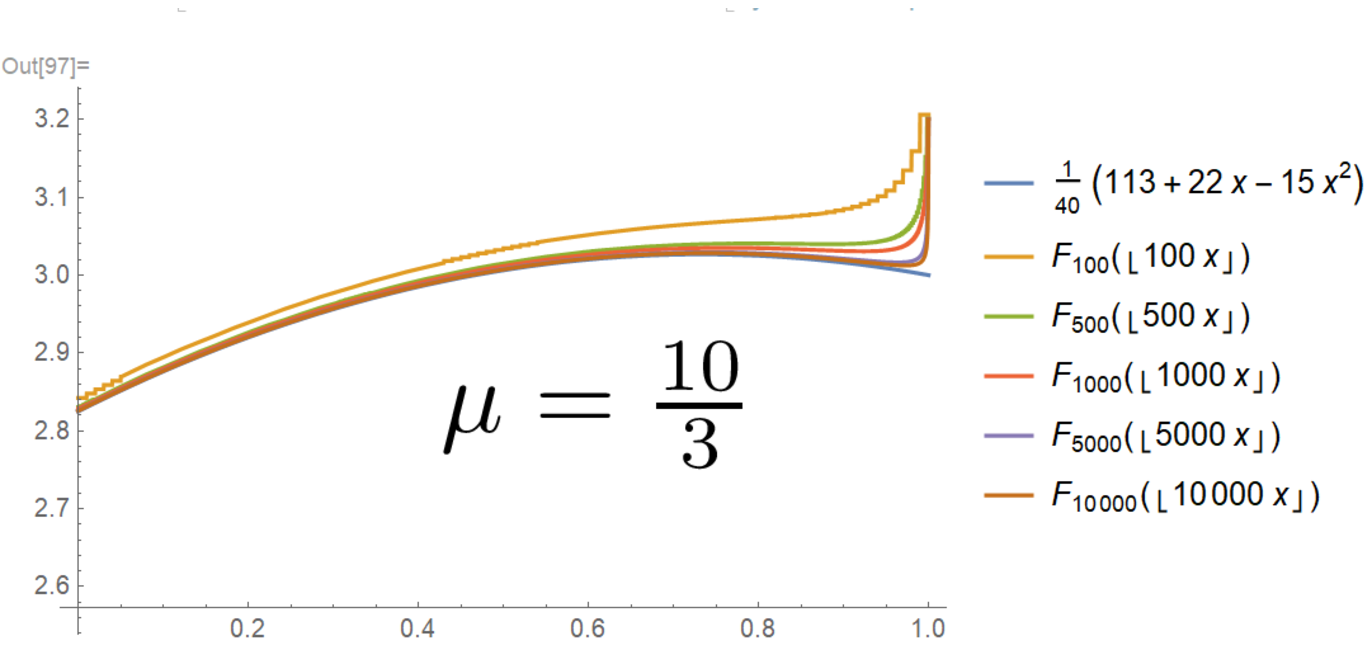}
    \caption{Pointwise, but not uniform, convergence in [0,1) for $\mu=10/3$ in Example \ref{exmu}.}
		\label{fig:mu103}
\end{figure}
\end{ejem}

\begin{ejem}
  Let now
  \begin{align*}
  F_{n}(k)&=G_{n}(k)+H_{n}(k)F_{n}(k+1),\\
	F_{n}(n)&=\mu
  \end{align*}
  where
  $$
  G_n(k)=\left(\frac{k}{n}\right)^n+\frac{1}{k+n}; \textrm{   }H_n(k)=\frac{k}{k+1}.
  $$
In this case,
$$
G_n(n-1)+\mu H_n(n-1)= \left(\frac{n-1}{n}\right)^n+\frac{1}{2 n-1}-\frac{\mu
}{n}+\mu
$$
so that $\lim_n G_n(n-1)+\mu H_n(n-1)=\mu+e^{-1} \neq \mu$ and condition (2) in Theorem \ref{main} does not hold for any value of $\mu$.

Let us check Conjecture \ref{con:two}. First, note that the following limit exists
$$
\Theta:= \lim_{n \rightarrow \infty} \lim_{k \rightarrow\infty} F_n(n-k)= \frac{1}{e-1}+\mu.
$$
The solution of the differential IVP:
$$
y'(x)=- \frac{1}{x+1}+ \frac{y(x)}{x};\quad y(1)=\Theta
$$
is:
$$
f(x)=x \left(\frac{\log (x)-e \log (2 x)+(e-1) \mu +1+\log    (2)}{e-1}+\log (x+1)\right).
$$

In Figure \ref{fig:pato2} one can perceive the expected punctual convergence to $f(x)$ in $(0,1)$, as conjectured.
\begin{figure}[h]
  \begin{center}\includegraphics[width= 5in,height= 2.0in]{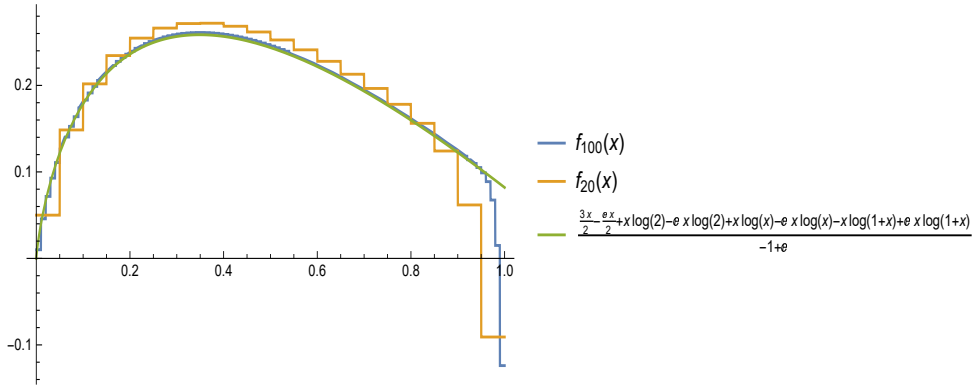}
    \caption{ $f_{n}(x)=F_n(\flr{nx}$ for  $n \in \{20,100\}$ and its limit
      $f(x)$ for $\mu=-1/2$.}\label{fig:pato2}
  \end{center}
\end{figure}

Let $\vartheta=0.34873760521\dots$ be the value at which $f(x)$ reaches its maximum in $[0,1]$. Notice the following approximations:
$$
\arg \max \{ F_{10^5}(k): 0<k<10^{5}\}=34873\approx 10^5\cdot \vartheta=34873.760521\dots
$$
and
$$
\max \{ F_{10^5}(k): 0<k<10^{5}\}=0.25856851103\ldots  \approx f(\vartheta) = 0.25856593889\dots
$$
\end{ejem}

\subsection{Piecewise functions. Gusein-Zade's generalized version of the secretary problem}

There are cases in which the optimal strategy has two (or more)
thresholds. In these cases Theorem \ref{main} and Proposition \ref{nuevaa}
can only provide the asymptotic value of the last one. The adaptation
of both results to this case is not straightforward but the idea looks promising. The following result, resembling Theorem \ref{inicial}, holds in any case.

\begin{prop}
  \label{medio} Let $\{\mathbf{s}_{n}\}_{n\in\mathbb{N}}$ with
  $\mathbf{s}_{n} \in\{0,...,n\}$ be such that
  $\lim_{n\rightarrow\infty}\frac{\mathbf{s}_{n}}{n}= \mathbf{s}$ are
  the real sequences of functions
  $\{F_{n}\}_{n\in \mathbb{N}}$, $\{G_{n}^{1}\}_{n\in\mathbb{N}}$,
  $\{G_{n}^{2}\}_{n\in \mathbb{N}}$, $\{H_{n}^{1}\}_{n\in\mathbb{N}}$
  and let $\{H_{n}^{2}\}_{n\in \mathbb{N}}$ defined in $\{0,...,n\}$,
  satisfy:

\[%
\begin{array}
  [c]{ccl}
  F_{n}(k) & = & G_{n}^{1}(k)+H_{n}^{1}(k)F_{n}(k+1) \text{ if }
                 k<\mathbf{s}%
_{n} \vspace{1ex}\\
  F_{n}(k) & = & G_{n}^{2}(k)+H_{n}^{2}(k)F_{n}(k+1) \text{ if }
                 \mathbf{s}_{n}
                 \leq k<n \vspace{1ex}\\
F_{n}(n) & = & \mu.
\end{array}
\]

Given $x\in\mathbb{R}$, define
\[
f_{n}(x):=F_{n}(\lfloor{nx}\rfloor)
\]
\[
h_{n}^{i}(x):=n(1-H_{n}^{i}(\lfloor{nx}\rfloor))
\]
\[
g_{n}^{i}(x):=nG_{n}^{i}(\lfloor{nx}\rfloor).
\]

If the following conditions hold:
\begin{itemize}
\item[i)] The sequences $\{h_{n}^{1}\}$ y $\{g_{n}^{1}\}$ converge
  punctually puntualmente in $(0,\mathbf{s}]$ and uniformly in
  $[\varepsilon,\mathbf{s}]$ for any $0<\varepsilon<\mathbf{s}$ to
  the continuous functions $h^{1}(x)$ y $g^{1}(x)$, respectively.
\item[ii)] The sequences $\{h_{n}^{2}\}$ y $\{g_{n}^{2}\}$ converge
  punctually in $[\mathbf{s},1)$ and uniformly in
  $[\mathbf{s},\varepsilon]$ for any $\mathbf{s}<\varepsilon<1$ to
  the continuous functions $h^{2}(x)$ y $g^{2}(x)$, respectively.
\item[iii)] The sequence $\{f_{n}\}$ converges uniformly in $[0,1]$ to
  a continuous function $f$.
\end{itemize}

Then: $f(1)=\mu$, and $f$ is the solution, in $[\mathbf{s},1]$ of the initial value problem:
\[
y'(x)-g^{2}(x); y(1)=\mu
\]
and, $f$ is also the solution in $(0,\mathbf{s}]$ of the IVP
\[
y'(x)= y(x) g^{1}(x) -g^{1}(x),\, y(\mathbf{s})=f(\textbf{s}).
\]
\end{prop}
The proof of this result is identical to the one in \cite{yo}, but it presents the exact same weakness. Namely, the required assumption of the uniform convergence of $f_n$. Our approach is to find conditions analogue to those of this paper (cf. Theorem \ref{main}) eliminating that requirement.

In what follows, we assume that such a result exists in order to explain how the secretary problem in which success is reached upon choosing either the best or the second-best candidate would be studied (asymptotically). This variant has already been studied by Gilbert and Mosteller \cite{gil}, and by Gusein-Zade \cite{gus}. The following result gives the optimal strategy.

\begin{prop}\label{dosum}
  For any $n\in \mathbb{N}$ there are $r_n, s_n  \in [0,n]$ such that the following strategy is optimal
  \begin{enumerate}
  \item Do not choose any candidate up to interview $r_n$.
  \item From interview $r_n$ to $s_n$ (inclusive),
    choose the first candidate which is better than the previous ones.
  \item After interview $s_n$, choose the first candidate
    which is at least the second-best among the already interviewed.
  \end{enumerate}
\end{prop}

Let $S_n(k)$ be the success probability when choosing the candidate in the $k$-th interview, assuming it is the second-best among the interviewed ones. Certainly,
$$
S_n(k)=\frac{\binom{k}{2}}{\binom{n}{2}}.
$$
Let $M_n(k)$ be the success probability when choosing the candidate in the $k$-th interview, assuming it is the best among the interviewed ones. The following recurrence holds:
$$
M_n(k)=\frac{1}{k+1} S_n(k+1)+\frac{k}{k+1}M_n(k+1); M_n(n)=1.
$$
From the above follows that
$$
M_n(k)=\frac{k^2-2 k n+k}{n-n^2}.
$$

Let $\overline{F}_n(k)$ be the success probability after rejecting the
candidate in the $k$-th interview, and waiting to choose the first
which is at least second-best among the already interviewed. We have:
$$
\overline{F}_n(k)= \frac{2}{n}+
\frac{k-1}{k+1}\overline{F}_{n}(k+1);\overline{F}_n(n)=0 \Longrightarrow
\overline{F}_n(k)=-\frac{2 k (k-n)}{(n-1) n}
$$

\subsubsection{Computing $\lim_n\frac{s_n}{n}$}
This can truly be done using the results of this paper. Notice that
the optimal threshold $s_n$ is the last value of $k$ for
which rejecting a second-best candidate (among the
interviewed ones) is preferable to choosing her. That is, $s_n$ satifies that
$$
S(s_n)
= \frac{s_n^2-s_n}{(n-1) n}<F_n(s_n)
$$
and
$$
F_n(s_n+i) \leq S(s_n+i)
=\frac{(s_n+i)^2-s_n-i}{(n-1) n}.
$$

We know from the formula for $\overline{F}_n(k)$ that $\overline{f}_n(x):=\overline{F}(\lfloor n x\rfloor)$
converges uniformly in $[0,1]$ to $\overline{f}(x):=2 (x - x^2)$. In addition, it is trivial to verify that
$S(\lfloor n x\rfloor)$ converges in $[0,1]$ to
$s(x)=x^2$. Hence, by Proposition \ref{nuevaa},
$\lim \frac{s_n}{n}=\frac{2}{3}$, which is the largest
solution of the equation $s(x)=\overline{f}(x)$ in $[0,1]$.

\subsubsection{Computing $\lim_n\frac{r_n}{n}$ and the asymptotic probability of success.}
Since $s_n$ is the second optimal threshold (Proposition \ref{dosum}), we denote by $F_n(k)$ the probability of success when rejecting the $k$-th candidate and waiting to:
\begin{enumerate}
\item Choose the first one which is the best among the interviewed
  ones if this happens before the $s_n$-th interview, or
\item Choose the one which is at least second-best if this happens
  after the $s_n$-th interview.
\end{enumerate}
Equivalently, $F_n(k)$ represents the probability of success when using the first threshold, if $k\leq s_n$, and if $k>s_n$, then $F_n(k)$ is the probability of success when rejecting the $k$-th candidate, waiting to choose one which is at least second-best afterwards (i.e. $F_n(k)=\overline{F}_n(k)$ for $k>r_n$). In other words,
$$
F_{n}(k)=
\left\{
  \begin{array}{ccccc}
                   \frac{2}{n}+ \frac{k-1}{k+1}F_{n}(k+1) &
                                                            \text{ if}
                   & k<s_{n} \vspace{1ex}\\
                   \frac{ M_n(k+1)}{k+1}+ \frac{k}{k+1}F_{n}(k+1) &\text{ if } &
                                                                                 s_{n} \leq k<n \vspace{1ex} \\
                   0 & \text{ if }& k=n
  \end{array}
\right.
$$

Now, either assuming the uniform convergence in $[0,1]$ of $F_n(\lfloor n x\rfloor )$ to $f(x)$, or assuming some kind of generalization of Theorem \ref{main}, one would reason as follows.
Consider the initial value problem
$$ y'(x)=\frac{y(x)}{x}+x-2,\quad
y\left(\frac{2}{3}\right)
=
\overline{f}\left(\frac{2}{3}\right)
=
\frac{4}{9}
$$
whose solution is
$$
f(x)= x^2-2 x \log (x)-2 x \log \left(\frac{3}{2}\right).
$$
This, together with the previous computation for $[4/9,1]$ gives:
$$
f(x)=
\left
  \{
  \begin{array}{lcc}
    x^2-2 x \log (x)-2 x \log \left(\frac{3}{2}\right) &
                                                         \text{ if }&                                                                     0\leq x \leq \frac{2}{3} \\
    -2 \left(x^2-x\right)  &  \text{ if }
                                                                    & \frac{2}{3} < x \leq 1 \\
  \end{array}
\right.
$$

The maximum of $f(x)$ in $[0,1]$ is reached at $\vartheta=-W\left(-\frac{2}{3 e}\right)=0.34698160970757\ldots$, so that
$$
\lim \frac{r_n}{n}=\vartheta=
-W\left(-\frac{2}{3 e}\right)=0.34698160970757\ldots,
$$
and
$$
\lim \mathbf{P}_n=f(\vartheta)=\vartheta (2 -  \vartheta )=
0.57356698193989\ldots$$
An these values coincide with the solutions from \cite{gil} and \cite{gus}.

\end{document}